\documentclass[a4paper,12pt]{scrartcl}

\usepackage[utf8]{inputenc}
\usepackage[english]{babel} 
\usepackage{amsmath,bm}
\usepackage{amssymb}
\usepackage{amsthm}
\usepackage{amsfonts}
\usepackage[mathscr]{eucal}
\usepackage{graphicx}
\usepackage{subfigure}
\usepackage{url}
\usepackage{xcolor}
\usepackage{algorithmic}
\usepackage{algorithm,booktabs}
\usepackage[mathscr]{eucal}
\usepackage[labelfont=bf]{caption}
\usepackage{tikz}
\usepackage{nicefrac}
\usepackage{apptools}
\usepackage{chngcntr}
\usepackage{multirow}
\usepackage{array}
\usepackage{siunitx}
\usepackage{adjustbox}
\usepackage[parfill]{parskip}
\usepackage{hyphenat}
\usepackage[all]{nowidow}

\newcommand{\bx}{{\boldsymbol x}}
\newcommand{\bn}{{\boldsymbol n}}
\newcommand{\bs}{{\boldsymbol s}}

\newcommand{\dx}{{\,\mathrm d\bx}}
\newcommand{\ds}{{\,\mathrm d\bs}}

\newcommand{\ubar}{\bar{u}}
\newcommand{\util}{\tilde{u}}

\newcommand{\yid}{y^{\mathsf{id}}}
\newcommand{\yidshift}{\tilde{y}^{\mathsf{id}}}

\newcommand{\ynadwb}{y^{\mathsf{nad,(w)}}}

\newcommand{\ynadIw}{y^{\textsl{nad,I,w}}}

\newcommand{\tD}{t^D}
\newcommand{\tDI}{t^{D^I}}

\newcommand{\R}{\mathbb{R}}
\newcommand{\N}{\mathbb{N}}

\newcommand{\Us}{\mathscr{U}}

\newcommand{\Xad}{\mathscr{X}_{\mathsf{ad}}}

\newcommand{\Uad}{\mathscr{U}_{\mathsf{ad}}}
\newcommand{\Ueq}{\mathscr{U}_{\mathsf{eq}}}

\newcommand{\Uopt}{\mathscr{U}_{\mathsf{opt}}}
\newcommand{\Uoptw}{\mathscr{U}_{\mathsf{opt,w}}}
\newcommand{\Utiloptw}{\tilde{\mathscr{U}}_{\mathsf{opt,w}}}

\newcommand{\Uoptwnum}{\mathscr{U}_{\mathsf{opt,w}}^{\mathsf{num}}}

\newcommand{\Uoptloc}{\mathscr{U}_{\mathsf{opt,loc}}}
\newcommand{\Uoptlocb}{\mathscr{U}_{\mathsf{opt,(loc)}}}
\newcommand{\Uoptwloc}{\mathscr{U}_{\mathsf{opt,w,loc}}}
\newcommand{\Uoptwbloc}{\mathscr{U}_{\mathsf{opt,(w),loc}}}

\newcommand{\Uoptwblocb}{\mathscr{U}_{\mathsf{opt,(w),(loc)}}}
\newcommand{\Jopt}{\mathscr{J}_{\mathsf{opt}}}
\newcommand{\Joptw}{\mathscr{J}_{\mathsf{opt,w}}}

\newcommand{\Joptloc}{\mathscr{J}_{\mathsf{opt,loc}}}
\newcommand{\Joptwloc}{\mathscr{J}_{\mathsf{opt,w,loc}}}

\newcommand{\Joptwblocb}{\mathscr{J}_{\mathsf{opt,(w),(loc)}}}

\newcommand{\Zopt}{\mathscr{Z}_{\mathsf{opt}}}

\newcommand{\JJ}{\mathcal{J}}

\newcommand{\Jhat}{\hat{J}}

\newcommand{\Qoptw}{\mathcal{Q}_{\mathsf{opt,w}}}

\newcommand{\Qoptlocb}{\mathcal{Q}_{\mathsf{opt,(loc)}}}
\newcommand{\Qoptwloc}{\mathcal{Q}_{\mathsf{opt,w,loc}}}

\newcommand{\Qoptwlocb}{\mathcal{Q}_{\mathsf{opt,w,(loc)}}}

\DeclareOldFontCommand{\bf}{\normalfont\bfseries}{\mathbf}
\DeclareOldFontCommand{\rm}{\normalfont\rmfamily}{\mathrm}
\newtheorem{Definition}{Definition}

\newtheorem{Assumption}{Assumption}
\newtheorem{Remark}{Remark}

\newtheorem{Lemma}{Lemma}
\newtheorem{Theorem}{Theorem}
\newtheorem{Corollary}{Corollary}

\begin{document}
\title{A trust region reduced basis Pascoletti-Serafini algorithm for multi-objective PDE-constrained parameter optimization}


\author{Stefan Banholzer\footnote{University of Konstanz, Universit\"atsstra\ss e 10, 78464 Konstanz, Germany; \{stefan.banholzer,luca.mechelli,stefan.volkwein\}@uni-konstanz.de}, Luca Mechelli$^*$ and Stefan Volkwein$^*$}

\date{}
\maketitle
\begin{abstract}
In the present paper non-convex multi-objective parameter optimization problems are considered which are governed by elliptic parametrized partial differential equations (PDEs). To solve these problems numerically the Pascoletti-Serafini scalarization is applied and the obtained scalar optimization problems are solved by an augmented Lagrangian method. However, due to the PDE constraints, the numerical solution is very expensive so that a model reduction is utilized by using the reduced basis (RB) method. The quality of the RB approximation is ensured by a trust-region strategy which does not require any offline procedure, where the RB functions are computed in a greedy algorithm. Moreover, convergence of the proposed method is guaranteed. Numerical examples illustrate the efficiency of the proposed solution technique.
\end{abstract}

\section{Introduction}

Multi-objective optimization plays an important role in many applications, e.g., in industry, medicine or engineering. One of the mentioned examples is the minimization of costs with simultaneous quality optimization in production or the minimization of CO$_2$ emission in energy generation and simultaneous cost minimization. These problems lead to multiobjective optimization problems (MOPs), where we want to achieve an optimal compromise with respect to all given objectives at the same time. Normally, the different objectives are contradictary such that there exists an infinite number of optimal compromises. The set of these compromises is called the \emph{Pareto set}. The goal is to approximate the Pareto set in an efficient way, which turns out to be more expensive than solving a single objective optimization problem. 

Since MOPs are of great importance, there exist several algorithms to solve them. Among the most popular methods are scalarization methods, which transform MOPs into scalar problems. For example, in the weighted sum method \cite{Ehr05,Mie99,Zad63}, convex combinations of the original objectives are optimized. However, in our case the multi-objective optimization problem
\begin{align}
\min\Jhat(u)=\big(\Jhat_1 (u),\ldots,\Jhat_k (u)\big)^T\quad\text{subject to (s.t.)}\quad u \in \Uad
\tag{\textbf{MOP}}
\label{MOP-Intro}
\end{align}
is non-convex with a bounded, non-empty, convex and closed set $\Uad$. To solve \eqref{MOP-Intro} a suitable scalarization method in that case is the Pascoletti-Serafini (PS) scalarization \cite{Eic08,Pas84}: For a chosen \emph{reference point} $z \in \mathbb{R}^k$ and a given \emph{target direction} $r \in \mathbb{R}^k$ with $r>0$ the \textup{Pascoletti-Serafini problem} is given by
\begin{equation}
\min t\quad\text{s.t.}\quad (t,u) \in \mathbb R\times\Uad\text{ and }\Jhat(u) - z\leq t \, r.
\label{PSP-Intro}
\tag{$\textbf{P}^\mathsf{PS}_{\bm{z,r}}$}
\end{equation}
In the present paper \eqref{PSP-Intro} is solved by an augmented Lagrangian approach. However, in our case the evaluation of the objective $\Jhat$ requires the solution of an elliptic partial differential equation (PDE) for the given parameter $u$. This implies further that for the computation of the gradients $\nabla\Jhat_i$, $i=1,\ldots,k$, adjoint PDEs have to be solved; cf. \cite{HPUU09}. Here, surrogate models offer a promising tool to reduce the computational effort significantly \cite{SvdVR08}. Examples are dimensional reduction techniques such as the Reduced Basis (RB) method \cite{HRS16,PR07}. In an offline phase, a low-dimensional surrogate model of the PDE is constructed by using, e.g., the greedy algorithm, cf. \cite{HRS16,IUV17,BGRV21}. In the online phase, only the RB model is used to solve the PDE, which saves a lot of computing time. 

We propose an extension of the method in \cite{Ban21} for solving multi-objective PDE-constained parameter optimization problems. This procedure is based on a combination of a trust-region reduced basis method \cite{Ban22,Kei21} and the PS method. In particular, we discuss different strategies to handle the increasing number of reduced basis functions, which is crucial in order to guarantee good performances of the algorithm.

The paper is organized as follows: In Section~\ref{Section:2} we introduce a general MOP and explain the PS method, in particular, a hierarchical version of the PS algorithm which turns out to be very efficient in the numerical realization. The concrete PDE-constrained MOP is investigated in Section~\ref{Section:PDE:model}. The trust-region RB method and its combination with the PSM is described in Section~\ref{Section:4}. Convergence is ensured and the algorithmic realization of the approach is explained. Numerical examples are discussed in detail in Section~\ref{Section:numerics}. Finally, we draw some conclusions.

\section{Multi-objective optimization}
\label{Section:2}

Let $(\Us,\langle \cdot\,, \cdot \rangle_\Us)$ be a real Hilbert space, $\Uad \subset \Us$ non-empty, convex and closed, $k \geq 2$ arbitrary and $\Jhat_1,\ldots,\Jhat_k \colon \Uad \subset \Us \to \mathbb{R}$ be given real-valued functions. In this manuscript, we assume also that $\Uad$ is bounded. This is an assumption we will require later for the convergence of our method. Note that one can derive similar results of this section if $\Uad$ is unbounded by introducing additional assumptions; cf. \cite{Ban21}. To shorten the notation, we write $\Jhat := (\Jhat_1,\ldots,\Jhat_k)^T \colon \Uad \to \mathbb{R}^k$. In the following, we deal with the multi-objective optimization problem
\begin{align}
\min\Jhat(u)\quad\text{s.t.}\quad u \in \Uad.
\tag{\textbf{MOP}} \label{Equation:MultiobjectiveOptimizationProblem}
\end{align}

\begin{Definition} \label{Definition:AdmissibleSetObjectiveSet}
	\begin{enumerate}
		\item[\em a)] The functions $\Jhat_1,\ldots,\Jhat_k$ are called \emph{cost} or \textup{objective functions}. Analogously, the vector-valued function $\Jhat \colon \Uad \to \R^k$ is named the \emph{(multi-objective) cost} or \emph{(multi-objective) objective function}.
		\item[\em b)] The Hilbert space $\Us$ is named the \emph{admissible space}, the set $\Uad$ is called the \textup{admissible set} and a vector $u \in \Uad$ is called \textup{admissible}. 
		\item[\em c)] The space $\R^k$ is named the \emph{objective space} and the image set $\Jhat(\Uad)$ is called the \emph{objective set}. A vector $y = \Jhat(u) \in \Jhat(\Uad)$ is called \emph{objective point}.
	\end{enumerate}
\end{Definition}

\begin{Definition}[Partial ordering on $\R^k$]
	\label{Definition:ProductOrderingOnRk}
	On $\mathbb{R}^k$ we define the partial ordering $\leq$ as
	\begin{align*}
	x \leq y & :\iff \left( \forall i \in \{1,\ldots,k\} \colon x_i \leq y_i \right)
	\end{align*}
	for all $x,y \in \mathbb{R}^k$. Moreover, we define
	\begin{align*}
	x < y & :\iff \left( \forall i \in \{1,\ldots,k\} \colon x_i < y_i \right).
	\end{align*}
	For convenience, we write 
	\begin{align*}
	x \lneqq y & :\iff \left( x \leq y \;\; \& \;\; x \neq y \right)
	\end{align*} 
	for all $x,y \in \mathbb{R}^k$ and define the two sets $\mathbb{R}^k_\leq := \{ y \in \mathbb{R}^k \mid y \leq 0 \}$, $\mathbb{R}^k_\lneqq := \{ y \in \mathbb{R}^k \mid y \lneqq 0 \}$. Analogously, the relations $\geq$,  $>$ and $\gneqq$ as well as the sets $\mathbb{R}^k_\geq$ and $\mathbb{R}^k_\gneqq$ are defined.
\end{Definition}

\begin{Definition}[Pareto optimality]
	\label{Definition:ParetoOptimality}
	\begin{enumerate}
		\item  [\em a)] An admissible vector $\bar{u} \in \Uad$ and its corresponding objective point $\bar{y} := \Jhat(\ubar) \in \Jhat(\Uad)$ are called \emph{(locally) weakly Pareto optimal} if there is no $\tilde{u} \in \Uad$ (in a neighborhood of $\bar{u}$) with $\Jhat(\util) < \Jhat(\ubar)$. The sets
		\begin{align*}
		\Uoptw & := \{ u \in \Uad \mid u \text{ is weakly Pareto optimal} \} \subset \Uad, \\
		\Uoptwloc & := \{ u \in \Uad \mid u \text{ is locally weakly Pareto optimal} \} \subset \Uad
		\end{align*}
		are said to be the \emph{weak Pareto set} and the \textup{locally weak Pareto set}, respectively. The sets
		\begin{align*}
		\Joptw:= \Jhat(\Uoptw) \subset \mathbb{R}^k,\quad\Joptwloc:= \Jhat(\Uoptwloc) \subset \mathbb{R}^k,
		\end{align*} 
		are the \emph{weak Pareto front} and the \emph{locally weak Pareto front}, respectively.
		\item  [\em b)] An admissible vector $\ubar \in \Uad$ and its corresponding objective point $\bar{y} := \Jhat(\ubar) \in \Jhat(\Uad)$ are called \emph{(locally) Pareto optimal} if there is no $\tilde{u} \in \Uad$ (in a neighborhood of $\bar{u}$) with $\Jhat(\util) \lneqq \Jhat(\ubar)$.
		The sets 
		\begin{align*}
		\Uopt & := \{ u \in \Uad \mid u \text{ is Pareto optimal} \} \subset \Uad, \\
		\Uoptloc & := \{ u \in \Uad \mid u \text{ is locally Pareto optimal} \} \subset \Uad
		\end{align*}
		are called the \emph{Pareto set} and the \emph{local Pareto set}, respectively. The sets
		\begin{align*}
		\Jopt:=\Jhat(\Uopt) \subset \mathbb{R}^k,\quad\Joptloc:=\Jhat(\Uoptloc) \subset \mathbb{R}^k
		\end{align*} 
		are called the \emph{Pareto front} and the \emph{local Pareto front}, respectively.
	\end{enumerate}
\end{Definition}

If we talk about the different notions of (local) (weak) Pareto optimality in one sentence, we use the notation $\Uoptwblocb$ to keep the sentence compact. Analogously, $\Uoptwbloc$, $\Uoptlocb$, $\Joptwblocb$ etc.~are to be understood. An example with the different concept of Pareto optimality can be found in \cite[Example~1.2.6]{Ban21}.

The next theorem goes back to \cite{Bor83}. It also appears in a similar form in \cite{Har78,Saw85}.

\begin{Theorem} \label{Theorem:ExistenceOfParetoOptimalPoints}
	Suppose that there is $y \in \Jhat(\Uad) + \R^k_\geq$ such that the set $(y - \R^k_\geq) \cap (\Jhat(\Uad) + \R^k_\geq)$ is compact. Then it holds $\Jopt \neq \emptyset$.
\end{Theorem}

\begin{proof}
	This is a slight generalization of \cite[Theorem 2.10]{Ehr05} using the argument that adding $\R^k_\geq$ to the set $\Jhat(\Uad)$ does not change the Pareto front $\Jopt$.\hfill\quad
\end{proof}

Given any $y = \Jhat(u) \in \Jhat(\Uad)$ with $y \not\in \Jopt$, it follows directly from the definition of Pareto optimality that there is $\bar{y} = \Jhat(\bar{u}) \in \Jhat(\Uad)$ with $\bar{y} \lneqq y$. However, even if the Pareto front $\Jopt$ is not empty (e.g., since the assumptions of Theorem~\ref{Theorem:ExistenceOfParetoOptimalPoints} are satisfied), it is not clear that there is $\bar{y} \in \Jopt$ with $\bar{y} \lneqq y$. If this property holds for all $y \in \Jhat(\Uad) \setminus \Jopt$, the set $\Jopt$ is said to be \emph{externally stable}; cf. \cite{Ehr05,Saw85}.  

\begin{Definition} \label{Definition:ExternallyStable}
	The set $\Jopt$ is said to be \textup{externally stable} if for every $y \in \Jhat(\Uad)$ there is $\bar{y} \in \Jopt$ with $\bar{y} \leq y$. This is equivalent to $\Jhat(\Uad) \subset \Jopt + \R^k_\geq$.
\end{Definition}

Especially for the investigation of suitable solution methods for solving \eqref{Equation:MultiobjectiveOptimizationProblem}, we are interested in guaranteeing that the Pareto front is externally stable. The next result pro\-vides a sufficient condition for this property.

\begin{Theorem} \label{Theorem:ParetoFrontIsExternallyStable}
	If for every $y \in \Jhat(\Uad) + \R^k_\geq$ the set $(y - \R^k_\geq) \cap (\Jhat(\Uad) + \R^k_\geq)$ is compact, then $\Jopt$ is externally stable.
\end{Theorem}

\begin{proof}
	For a proof of a similar version of this theorem, we refer to \cite[Theorem 2.21]{Ehr05}.
\end{proof}
Among the methods to solve multi-objective optimization problems, the ones based on scalarization techniques are frequently appearing in the literature. Let us mention here the weighted-sum method \cite{Ehr05,Zad63}, the Euclidian reference point method \cite{Wie80} and the PS method \cite{Eic08,Pas84}. Since in our case the set $\Jhat(\Uad)+\R^k_{\geq}$ is non-convex, we apply the PS method which is proven to be able to solve a non-convex \eqref{Equation:MultiobjectiveOptimizationProblem}.

\subsection{The PS method}
\label{Section:PSM}

For a chosen \emph{reference point} $z \in \mathbb{R}^k$ and a given \emph{target direction} $r \in \mathbb{R}^k_>$ the \textup{PS problem} is given by
\begin{equation}
\min t\quad\text{s.t.}\quad (t,u) \in \mathbb R\times\Uad\text{ and }\Jhat(u) - z\leq t \, r.
\label{Equation:PascolettiSerafiniScalarization}
\tag{$\textbf{P}^\mathsf{PS}_{z,r}$}
\end{equation}
Analogously, we can define the PS problem as a scalarization problem. For $z \in \mathbb{R}^k$ and $r \in \mathbb{R}^k_>$ we define the scalarization function
\begin{align*}
g_{z,r} \colon \mathbb{R}^k \to \mathbb{R}, \quad x \mapsto g_{z,r}(x):=\max_{1\le i\le k}\frac{1}{r_i} \left(x_i - z_i \right),
\end{align*}
and the \emph{PS scalarized function} 
\[
\Jhat^{g_{z,r}}(u) := g_{z,r}(\Jhat(u))=\max_{1\le i\le k} \, \frac{1}{r_i} ( \Jhat_i(u) - z_i )\quad\text{for }u \in \Uad.
\] 
Then the \emph{reformulated PS problem} is given by
\begin{align}
\min\Jhat^{g_{z,r}}(u)\quad\text{s.t.}\quad u \in \Uad. \label{Equation:PascolettiSerafiniScalarization:Reformulated}
\tag{$\textbf{RP}^\mathsf{PS}_{z,r}$}
\end{align} 

The following theorem proved in \cite[Theorem~1.7.3]{Ban21} ensures the equivalence between \eqref{Equation:PascolettiSerafiniScalarization} and \eqref{Equation:PascolettiSerafiniScalarization:Reformulated}.

\begin{Theorem} \label{Theorem:PSM:EquivalenceOfDifferentProblemFormulations}
	Let $z \in \mathbb{R}^k$ and $r \in \R^k_>$ be arbitrary. On the one hand, if $(\bar{u},\bar{t})$ is a global (local) solution of \eqref{Equation:PascolettiSerafiniScalarization}, then $\bar{u}$ is a global (local) solution of \eqref{Equation:PascolettiSerafiniScalarization:Reformulated} with minimal function value $\bar{t}$. On the other hand, if $\bar{u}$ is a global (local) solution of \eqref{Equation:PascolettiSerafiniScalarization:Reformulated}, then $(\bar{u},\bar{t})$ with $\bar{t} := \max_{1\le i\le k} ( \Jhat_i(\bar{u}) - z_i )/r_i$ is a global (local) solution of \eqref{Equation:PascolettiSerafiniScalarization}. 
\end{Theorem}

\begin{Assumption} \label{Assumption:NonConvexMOP}
	The cost functions $\Jhat_1,\ldots,\Jhat_k$ are weakly lower semi-continuous and bounded from below. 
\end{Assumption}

\begin{Theorem}
	\label{Theorem:PSM:ExistenceOfParetoOptimalSolution}
	Let \rm Assumption~\ref{Assumption:NonConvexMOP} be satisfied and $z \in \mathbb{R}^k$ as well as $r \in \R^k_>$ be arbitrary. Then \eqref{Equation:PascolettiSerafiniScalarization:Reformulated} has a global solution $\bar{u} \in \Uopt$.	
\end{Theorem}

\begin{proof}
	A proof of this statement can be found in \cite[Corollary 1.7.12]{Ban21}. \hfill\quad 
\end{proof}

The previous result also shows that the existing global solution of \eqref{Equation:PascolettiSerafiniScalarization:Reformulated} belongs to the Pareto set. To guarantee a good reconstruction of the Pareto set by the PS method, one needs that, given a (weakly) Pareto optimal point, it is possible to choose the parameters $z$ and $r$ such that this point solves \eqref{Equation:PascolettiSerafiniScalarization:Reformulated}. This is stated in \cite[Theorem~1.7.13]{Ban21}, which we report here for clearness. 

\begin{Theorem} \label{Theorem:PSM:AllParetoOptimalsAreSolutions}
	Let $\bar{u} \in \Uoptw$ be arbitrary. Then for every $r \in \R^k_>$ and every $\bar{t} \in \mathbb{R}$ we have that $\bar{u}$ is a global solution of \eqref{Equation:PascolettiSerafiniScalarization:Reformulated} for the reference point $z := \Jhat(\bar{u}) - \bar{t} r$. If even $\bar{u} \in \Uopt$, any other global solution $\tilde{u}$ of \eqref{Equation:PascolettiSerafiniScalarization:Reformulated} satisfies $\Jhat(\tilde{u}) = \Jhat(\bar{u})$.
\end{Theorem}

\begin{Remark}
	We refer the reader to \emph{\cite[Lemma~1.7.15]{Ban21}} for the derivation of first-order necessary optimality condition for a global solution of \eqref{Equation:PascolettiSerafiniScalarization}.
\end{Remark}

Thus, the PS method can compute in principle every (locally) (weak) Pareto optimal point so that many algorithms based on PS method have been proposed. Here we only mention the ones which are related to (but differ from) our proposed technique. Our main idea is to keep the parameter $r$ fixed, while varying the reference point $z$. This was also proposed in \cite{Eic08}, but the method turns out to be not applicable numerically for $k>2$. In \cite{Gri09}, the authors provide assumptions on the Pareto front to ensure that the so-called trade-off limits (i.e., points on the Pareto front which cannot be improved in at least one component), are given by the solution to subproblems. Their idea was then to find these trade-off points first and then compute the rest of the Pareto front. A similar idea but with the use of Centroidal Voronoi Tessellation was presented by \cite{Mot12}. Finally, \cite{Kha15} shows and fixes some problematic behavior associated to the algorithm in \cite{Gri09}. We follow the idea of the mentioned contributions of hierarchically solving subproblems of \eqref{Equation:MultiobjectiveOptimizationProblem}, but with the focus of finding a set of reference points, by looking at subproblems, for which we can obtain Pareto optimal points. We are then not interested in finding 'boundary' points (i.e., the trade-off limits) of the Pareto front and then filling its 'interior' as in \cite{Gri09,Kha15,Mot12}, but rather to partly generalize this approach. In what follows, we characterize which reference points are necessary and/or sufficient for computing the entire (local) (weak) Pareto front. First we recall the following well-defined mappings; cf. \cite[Definition~1.7.16]{Ban21}.

\begin{Definition} \label{Definition:PSM:SolutionMappings}
	We define the set-valued mappings
	\begin{align*}
	\Qoptw & \colon \R^k \rightrightarrows \Uoptw, \, && z \mapsto \{ u \in \Uad \mid u \text{ is a global solution of } \eqref{Equation:PascolettiSerafiniScalarization:Reformulated} \}, \\
	\Qoptwloc & \colon \R^k \rightrightarrows \Uoptwloc, \, && z \mapsto \{ u \in \Uad \mid u \text{ is a local solution of } \eqref{Equation:PascolettiSerafiniScalarization:Reformulated} \}, \\
	\Qoptlocb & \colon \R^k \rightrightarrows \Uoptlocb, \, && z \mapsto \Qoptwlocb(z) \cap \Uoptlocb.
	\end{align*}
\end{Definition}

From Theorem~\ref{Theorem:PSM:EquivalenceOfDifferentProblemFormulations}, it follows that $\mathcal{Q}_{\mathsf{opt,(w),(loc)}}(\R^k) = \mathscr{U}_{\mathsf{opt,(w),(loc)}}$.
Furthermore, if Assumption~\ref{Assumption:NonConvexMOP} is satisfied, we infer from Theorem~\ref{Theorem:PSM:ExistenceOfParetoOptimalSolution} that $\mathcal{Q}_{\mathsf{opt,(w),(loc)}}(z) \neq \emptyset$ for all $z \in \R^k$.
We also introduce the notion of a (locally) (weakly) Pareto sufficient set for the PSM.

\begin{Definition}
	\label{Definition:PSM:ParetoSufficient}
	A set $Z \subset \mathbb{R}^k$ is called \emph{(locally) (weakly) Pareto sufficient} if we have $\mathcal{Q}_{\mathsf{opt,(w),(loc)}}(Z) = \mathscr{U}_{\mathsf{opt,(w),(loc)}}$.
\end{Definition}

Hence, a (locally) (weakly) Pareto sufficient set contains the reference points which allow us to compute the entire (local) (weak) Pareto front. Clearly, the set $\R^k$ is (locally) (weakly) Pareto sufficient, but this fact is not computationally useful. The next lemma gives a first condition towards this computational efficiency.
\begin{Lemma} \label{Lemma:SufficientConditionParetoSufficent}
	Let $Z \subset \mathbb{R}^k$ be arbitrary. $Z$ is (locally) (weakly) Pareto sufficient, if
	\begin{align}
	\forall \ubar \in \mathscr{U}_{\mathsf{opt,(w),(loc)}} \colon \exists t \in \mathbb{R} \colon \Jhat(\ubar) - tr \in Z. \label{Equation:PSM:SufficientConditionParetoSufficient}
	\end{align}
\end{Lemma}

\begin{proof}
	Let $Z \subset \mathbb{R}^k$ be such that \eqref{Equation:PSM:SufficientConditionParetoSufficient} holds. Let $\bar{u} \in \mathscr{U}_{\mathsf{opt,(w),(loc)}}$ be arbitrary. We need to show that there is a $z \in Z$ with $\bar{u} \in \mathcal{Q}_{\mathsf{opt,(w),(loc)}}(z)$. Indeed, by \eqref{Equation:PSM:SufficientConditionParetoSufficient} there is $t \in \mathbb{R}$ with $z := \Jhat(\bar{u}) - t r \in Z$ and by Theorem~\ref{Theorem:PSM:AllParetoOptimalsAreSolutions} we already have $\bar{u} \in \mathcal{Q}_{\mathsf{opt,(w),(loc)}}(z)$.\hfill\quad
\end{proof}

To proceed we introduce the concepts of ideal point and shifted ideal point, which will be used to define an optimal Pareto sufficient set\footnote{The word 'optimal' here means that removing any point from the set will cause the loss of the Pareto sufficient property.}.

\begin{Definition}
	\label{Definition:yidealpoint}
	\begin{enumerate}
		\item [\em a)] We define the \emph{ideal objective point} $\yid \in \mathbb{R}^k \cup \{-\infty\}$ by $\yid_i := \inf_{u \in \Uad} \Jhat_i(u)$ for all $i \in \{1,\ldots,k \}$.
		\item [\em b)] For an arbitrary vector $\tilde{d} \in \R^k_>$ define the \emph{shifted ideal point} $\yidshift := \yid - \tilde{d}$. Let $D_i \subset \mathbb{R}^k$ be given by $D_i := \{ y \in \mathbb{R}^k \mid y \geq \yidshift, \; y_i = \yidshift_i \} $ for all $i \in \{1,\ldots,k \}$. Then the set $D \subset \mathbb{R}^k$ is defined by $D := \bigcup_{i=1}^k D_i$.
		\item [\em c)] We define $\mathscr{Z}^D_{\mathsf{opt,(w),(loc)}} := \{ z \in D \mid \exists \ubar \in \mathscr{U}_{\mathsf{opt,(w),(loc)}} \colon \exists t \in \mathbb{R} \colon z = \Jhat(\ubar) - t r \}$.
		\item [\em d)] For any $y\in\mathbb{R}^k$ we set $t^D(y):= \min_{i\in\{1,\ldots,k\}} (y_i-\yidshift_i)/r_i\in\R$.
	\end{enumerate}
\end{Definition}

\begin{Remark}
	\label{RemarkZD}
	It is proved in \emph{\cite[Lemma~1.7.24]{Ban21}} that 
	\begin{align*}
	\mathscr{Z}^D_{\mathsf{opt,(w),(loc)}} = \big\{ \Jhat(\ubar) - \tD (\Jhat(\ubar)) \, r \, \big| \, \ubar \in \mathscr{U}_{\mathsf{opt,(w),(loc)}} \big\}.
	\end{align*}
	Furthermore, the set $\mathscr{Z}^D_{\mathsf{opt,(w),(loc)}}$ is (locally) (weakly) Pareto sufficient and there is a Lipschitz continuous bijection between $\Zopt^D$ and the Pareto front $\Jopt$. Unfortunately there is no bijection between $\mathscr{Z}^D_{\mathsf{opt,(w),(loc)}}$ and $\mathscr{J}_{\mathsf{opt,(w),(loc)}}$, but the set $\mathscr{Z}^D_{\mathsf{opt,(w),(loc)}}$ is still (locally) (weakly) Pareto sufficient. Therefore, it is anyway possible to use it for the computation of the Pareto front.
\end{Remark}

\subsection{Hierarchical PS method}

Due to Definition~\ref{Definition:yidealpoint} and Remark~\ref{RemarkZD} the set $\mathscr{Z}^D_{\mathsf{opt,(w),(loc)}}$ can only by computed once the set $\mathscr{U}_{\mathsf{opt,(w),(loc)}}$ is available. Clearly, this characterization of $\mathscr{Z}^D_{\mathsf{opt,(w),(loc)}}$ is not useful for a numerical algorithm. Fortunately, in \cite{Ban21,Low84} it is shown that the Pareto set has a hierarchical structure. This means that the (weak) Pareto front and the (weak) Pareto sets of \eqref{Equation:MultiobjectiveOptimizationProblem} are contained in the set of all (weak) Pareto fronts and (weak) Pareto sets of all of its subproblems. This particular structure of the Pareto set can be exploited to set up a hierarchical algorithm for obtaining a superset of $\mathscr{Z}^D_{\mathsf{opt,(w),(loc)}}$ without computing entirely the (local) (weak) Pareto set first.

\begin{Definition} \label{Definition:SubproblemOfMOP}
	For the index set $I \subset \{1,\ldots,k\}$ we denote by $\Jhat^I$ the multi-objective cost function $(\Jhat_i)_{i \in I} \colon \Uad \to \mathbb{R}^I$, and call the problem
	\begin{align}
	\min \Jhat^I(u)\quad\text{s.t.}\quad u \in \Uad
	\tag{$\textbf{MOP}_I$} \label{Equation:MultiobjectiveOptimizationProblem_Subproblem}
	\end{align}
	a \emph{subproblem of \eqref{Equation:MultiobjectiveOptimizationProblem}}. For $I,K \subset \{1,\ldots,k\}$ with $K \subset I$,
	\begin{enumerate}
		\item [\em a)] and for every $y \in \mathbb{R}^I$ we denote by $y^K := (y_i)_{i \in K} \in \mathbb{R}^K$ the canonical projection to $\mathbb{R}^K$.
		\item [\em b)] the set $\Uoptwblocb^I := \{ u \in \Uad \mid u \text{ is (loc.) (weak.) Pareto optimal for } \eqref{Equation:MultiobjectiveOptimizationProblem_Subproblem} \}$	denotes the \textup{(local) (weak) Pareto set} and the set $\Joptwblocb^I := \Jhat^I(\Uoptwblocb^I) \subset \mathbb{R}^I$ denotes the \emph{(local) (weak) Pareto front of the subproblem \eqref{Equation:MultiobjectiveOptimizationProblem_Subproblem}}.
		\item [\em c)] the \emph{(local) (weak) nadir objective point for the subproblem \eqref{Equation:MultiobjectiveOptimizationProblem_Subproblem}} is defined by
		\begin{align*}
		y_i^{\mathsf{nad},I,\mathsf{(w),(loc)}} := \sup \{ y_i \mid y \in \mathscr{J}^I_{\mathsf{opt},(w),(loc)} \}\quad\text{for all } i \in I.
		\end{align*}
	\end{enumerate}
\end{Definition}

\begin{Definition} \label{Definition:PascolettiSerafiniScalarizationSubproblem}
	Let $I \subset \{1,\ldots,k\}$ be arbitrary. For a given reference point $z \in \mathbb{R}^{|I|}$ and target direction $r \in \mathbb{R}^{|I|}_>$, we define the \textup{PS problem for \eqref{Equation:MultiobjectiveOptimizationProblem_Subproblem}} by
	\begin{equation}
	\begin{aligned}
	\min t\quad\text{s.t.}\quad(t,u)\in\mathbb R\times\Uad\text{ and }\Jhat^I(u) - z \le tr^I.
	\end{aligned}
	\label{Equation:PascolettiSerafiniScalarizationSubproblem}
	\tag{$\textbf{P}^\mathsf{PS}_{I,z,r}$}
	\end{equation}
	Again, it is possible to show that \eqref{Equation:PascolettiSerafiniScalarizationSubproblem} is equivalent (in the sense of {\rm Theorem~\ref{Theorem:PSM:EquivalenceOfDifferentProblemFormulations}}) to the problem 
	\begin{equation}
	\min \bigg(\max_{i \in I} \, \frac{1}{r_i} \left( \Jhat_i(u) - z_i \right)\bigg)\quad\text{s.t.}\quad u \in \Uad. \label{Equation:PascolettiSerafiniScalarizationSubproblem:Reformulated}
	\tag{$\textbf{RP}^\mathsf{PS}_{I,z,r}$}
	\end{equation} 
\end{Definition}

Let us mention that the statements proved in Section~\ref{Section:PSM} can be adapted for the PS mehod for the subproblems. Similarly we can also define the sufficient Pareto sets.

\begin{Definition} \label{Definition:ERPM:yidshiftAndD:Subproblem}
	Let $I \subset \{1,\ldots,k\}$ be arbitrary. Given the vector $\tilde{d} \in \R^k_>$ and the shifted ideal point $\yidshift \in \R^k$, which were both introduced in Definition~\emph{\ref{Definition:yidealpoint}}, let $D_i^I \subset \mathbb{R}^I$ be given by
	\[
	D_i^I := \big\{ y \in \mathbb{R}^I\,\big|\,y \geq (\yidshift)^I, \; y_i = \yidshift_i \big\}\quad\text{for }i \in I.
	\] 
	Then the set $D^I \subset \mathbb{R}^I$ is defined by $D^I := \bigcup_{i \in I} D_i$. Moreover, for all $K \subset \{1,\ldots,k\}$ we define the sets 
	\begin{align*}
	\mathscr{Z}^{D^I,K}_{\mathsf{opt,(w),(loc)}} & := \big\{ z \in D^I \,\big|\, \exists \ubar \in \mathscr{U}^K_{\mathsf{opt,(w),(loc)}} \colon \exists t \in \R \colon z = \Jhat^I(\ubar) - t r^I \big\}.
	\end{align*}
	To ease the notation, we write $\mathscr{Z}^{D^I}_{\mathsf{opt,(w),(loc)}} := \mathscr{Z}^{D^I,I}_{\mathsf{opt,(w),(loc)}}$. If $I = \{1,\ldots,k\}$ we set 
	$\mathscr{Z}^{D,K}_{\mathsf{opt,(w),(loc)}} := \mathscr{Z}^{D^I,K}_{\mathsf{opt,(w),(loc)}}$ and $\mathscr{Z}^{D}_{\mathsf{opt,(w),(loc)}} := \mathscr{Z}^{D^I,I}_{\mathsf{opt,(w),(loc)}}$. Finally, for any $y \in \R^I$ we set $\tDI(y) := \min_{i \in I} \frac{y_i - \yidshift_i}{r_i} \in \R$.
\end{Definition}

Note that Remark~\ref{RemarkZD} can be rewritten for the subproblems. It can be shown that the set $\mathscr{Z}^{D^I}_{\mathsf{opt,(w),(loc)}}$ can be computed by using the sets $\mathscr{U}^K_{\mathsf{opt,(w),(loc)}}$ for all $K \subsetneq I$. This procedure requires the assumption that the cost functions $\Jhat_1,\ldots,\Jhat_k$ are upper semi-continuous. Other very technical conditions are omitted to ease and shorten the presentation here. For a reader interested in the details we refer to \cite[Sec. 1.7.4.2-1.7.4.4]{Ban21}. Here we just give the necessary numerical condition in order to compute a numerical approximation of the set $\mathscr{Z}^{D^I}_{\mathsf{opt,(w),(loc)}}$.

\begin{algorithm}[h!]
	\caption{Solving \eqref{Equation:MultiobjectiveOptimizationProblem} numerically by the hierarchical PS method \label{Algorithm:PSM:HierarchicalAlgorithm:Numerical}}
	\begin{algorithmic}[1]
		
		\FOR{$j = 1:k$}
		\STATE Set $I := \{ j \}$;\;
		\STATE Compute $\Uoptwnum(I) = \{ u \mid u \text{ minimizes } \Jhat_j \}$;\; \label{LineOfAlgorithm:AlgorithmHierarchicalPSM:AddingSingleMinimizers:Numerical}
		\STATE Choose $\tilde{d}_j$, compute $\yid_j$ and set $\yidshift_j=\yid_j-\tilde d_j$;\;
		\STATE Set $\mathcal{U}\mathcal{T}\mathcal{Z}^{\mathsf{num}}(I) = \{ (u,\tilde{d}_j,\yidshift_j) \mid u \in \Uoptwnum(I) \}$;\;
		\ENDFOR
		\FOR{$i = 2:k$}
		\FORALL{$I \subset \{1,\ldots,k \}$ with $\left|I\right| = i$}
		\STATE Initialize $\Uoptwnum(I) = \bigcup_{K \subsetneq I} \Uoptwnum(K)$ and $\mathcal{U}\mathcal{T}\mathcal{Z}^{\mathsf{num}}(I) = \emptyset$;\;
		\STATE Compute the reference points $Z^{\mathsf{num}}(I) = \{z \in \mathcal{Z}^{h,I} \mid \neg \eqref{Equation:PSM:RemoveReferencePointsInAlgorithm:Numerical} \}$;\; \label{LineOfAlgorithm:AlgorithmHierarchicalPSM:GeneratingReferencePoints:Numerical}
		\WHILE{$Z^{\mathsf{num}}(I) \neq \emptyset$}
		\STATE Choose $z \in Z^{\mathsf{num}}(I)$ and remove $z$ from $Z^{\mathsf{num}}(I)$;\;
		\STATE Solve \eqref{Equation:PascolettiSerafiniScalarizationSubproblem}/\eqref{Equation:PascolettiSerafiniScalarizationSubproblem:Reformulated};\; \label{LineOfAlgorithm:AlgorithmHierarchicalPSM:ComputingPSMSolution:Numerical}
		\STATE Set $\Uoptwnum(I) \leftarrow \Uoptwnum(I) \cup \mathcal{Q}_{\mathsf{opt,w}}^I(z)$;\; \label{LineOfAlgorithm:AlgorithmHierarchicalPSM:AddingPSMSolution:Numerical} 
		\STATE Set \\
		$\mathcal{U}\mathcal{T}\mathcal{Z}^{\mathsf{num}}(I) \leftarrow \mathcal{U}\mathcal{T}\mathcal{Z}^{\mathsf{num}}(I) \cup \{ (\ubar,\bar{t},z) \mid (\ubar,\bar{t}) \text{ gl. sol. of } \eqref{Equation:PascolettiSerafiniScalarizationSubproblem} \}$;\;
		\STATE Add solutions of PSPs with respect to redundant reference points: Set \\
		$\mathcal{U}\mathcal{T}\mathcal{Z}^{\mathsf{num}}(I) \leftarrow \mathcal{U}\mathcal{T}\mathcal{Z}^{\mathsf{num}}(I) \cup \{ (\ubar,\bar{t},\tilde{z}) \mid (\ubar,\bar{t}) \text{ gl. } \text {sol.} \text{ of } \eqref{Equation:PascolettiSerafiniScalarizationSubproblem},$ \\
		\hfill $\tilde{z} \in Z^{\mathsf{num}}(I) \cap [z - (\bar{t} r^I - (\Jhat^I(\bar{u}) - z)),z] \}$;\; \label{LineOfAlgorithm:AlgorithmHierarchicalPSM:RemovingUnnecessaryReferencePoint1:Numerical} 
		\STATE Remove redundant reference points: Set \\
		$Z^{\mathsf{num}}(I) \leftarrow Z^{\mathsf{num}}(I) \setminus [z - (\bar{t} r^I - (\Jhat^I(\bar{u}) - z)),z] $ for all $\bar{u} \in \mathcal{Q}_{\mathsf{opt,(w)}}^I(z)$;\; \label{LineOfAlgorithm:AlgorithmHierarchicalPSM:RemovingUnnecessaryReferencePoint2:Numerical}
		\ENDWHILE
		\ENDFOR
		\ENDFOR
		\IF{\textsl{computeParetoFront == true}}
		\STATE Remove all $u\in \Uoptwnum(\{1,\ldots,k\})$ with $u \not\in \Uopt$ by a non-dominance test;\;
		\ENDIF
	\end{algorithmic}
\end{algorithm}

To do so, we introduce a grid on $D^I$ as follows

\begin{Definition} \label{Definition:PSM:NumericalImplementation:Grid}
	Let $I \subset \{1,\ldots,k\}$ be arbitrary. For a given grid size $h > 0$ and any $i \in I$, we define
	\begin{align*}
	\mathcal{Z}^{h,I}_i & := \left\{ z \in D_i^I \, \middle\vert \, \forall j \in I \setminus \{i\} \colon \left( \exists k \geq 0 \colon z_j = \yidshift_j + \frac{h}{2} + k h \right) \, \& \, \left( z_j \leq \ynadIw_j - \bar{t}^i r_j \right) \right \}.
	\end{align*}
	Furthermore, we set $\mathcal{Z}^{h,I} := \bigcup_{i \in I} \mathcal{Z}^{h,I}_i$. If $I = \{1,\ldots,k\}$, we write $\mathcal{Z}^{h} := \mathcal{Z}^{h,I}$.
\end{Definition}

The idea is to only choose reference points that lie on the grid $\mathcal{Z}^{h,I}$ and do not satisfy the condition
\begin{align}
\exists K \subsetneq I \colon \exists (\bar{u},\bar{t},\bar{z}) \in \mathcal{U}\mathcal{T}\mathcal{Z}^{\mathsf{num}}(K) \colon z^K = \bar{z}^K \,\; \& \,\; z^{I \setminus K} \geq \Jhat^{I \setminus K}(\bar{u}) - \bar{t} r^{I \setminus K}, \label{Equation:PSM:RemoveReferencePointsInAlgorithm:Numerical}
\end{align} 
where $\mathcal{U}\mathcal{T}\mathcal{Z}^{\mathsf{num}}(K)$ is a numerical approximation of $\mathcal{U}\mathcal{T}\mathcal{Z}(K)= \{ (u,\tilde{d}_j,\yidshift_j) \mid u \in \Utiloptw(I) \}$. An explanation for excluding points based on \eqref{Equation:PSM:RemoveReferencePointsInAlgorithm:Numerical} can be found in \cite{Ban21}. Finally, we describe the proposed numerical hierarchical PS method in Algorithm~\ref{Algorithm:PSM:HierarchicalAlgorithm:Numerical}. 
\begin{Remark}
	In \emph{\cite{Say00}}, the author introduce three different quality criteria for a scalarization method.
	\begin{enumerate}
		\item [\em a)] \textbf{Coverage:} Every part of the Pareto set and front has to be represented in the sets $\Uoptwnum$ and $\Joptw^\mathsf{num}$, respectively. This can be measured by
		\[
		\mathsf{cov}(\Joptwblocb) := \max_{\bar{y} \in \Joptwblocb} \, \min_{y \in \Joptwblocb^{\mathsf{num}}} \left\Vert \bar{y} - y \right\Vert.
		\]
		In the case of Algorithm~\emph{\ref{Algorithm:PSM:HierarchicalAlgorithm:Numerical}}, we have that $\mathsf{cov}(\Joptwblocb)= \mathcal{O}(h)$ (cf. \emph{\cite{Ban21}}).
		\item [\em b)] \textbf{Uniformity:} The points on the Pareto set and front should be distributed (almost) equidistantly; cf. \emph{\cite[Remark 1.7.69-b)]{Ban21}}.
		\item [\em c)] \textbf{Cardinality:} The number of points contained in the numerical approximation should be reasonable. In the case of Algorithm~\emph{\ref{Algorithm:PSM:HierarchicalAlgorithm:Numerical}} is not possible to estimate a-priori the number of elements computed by the method. It is possible to show a bound which can be computed when the nadir objective point $\ynadwb$ is known (cf. \emph{\cite[Remark 1.7.69-c)]{Ban21}}). 
	\end{enumerate}	
\end{Remark}

\section{The non-convex parametric PDE-constrained MOP}
\label{Section:PDE:model}

Before defining our exemplary MOP, we introduce the PDE model which will later serve as an equality constraint. Let $\Omega \subset \R^d$, $d \in \{2,3\}$, be a bounded domain with Lipschitz-continuous boundary $\Gamma = \partial \Omega$. Furthermore, let $\Omega_1,\ldots,\Omega_m$ be a pairwise disjoint decomposition of the domain $\Omega$ and set $\Gamma_i := \partial \Omega_i \cap \partial \Omega$ for all $i = 1,\ldots,m$. Then we are interested in the following elliptic diffusion-reaction equation with Robin boundary condition:
\begin{subequations}\label{Equation:EllipticModelProblem:StateEquation}
	\begin{align}
	\label{Equation:EllipticModelProblem:StateEquation-1}
	- \nabla \cdot \left( \sum_{i=1}^{m} u_i^{\kappa} \, \chi_{\Omega_i} (\bx) \nabla y(\bx) \right) + u^r \, r(\bx) y(\bx) & = f(\bx) && \mbox{a.e. in } \Omega,\\
	\label{Equation:EllipticModelProblem:StateEquation-2}
	u_i^{\kappa} \,\frac{\partial y}{\partial \bn}(\bs) + \alpha y(\bs) &= \alpha y_a(\bs) &&\mbox{a.e. on }\Gamma_i.
	\end{align}
\end{subequations}
For every $i \in \{1,\ldots,m\}$, the parameter $u_i^\kappa > 0$ represents the diffusion coefficient on the subdomain $\Omega_i$. By $r \in L^\infty(\Omega)$, we denote a reaction function, which is supposed to satisfy $r > 0$ a.e.~in $\Omega$ and is controlled by the scalar parameter $u^r > 0$. On the right-hand side of \eqref{Equation:EllipticModelProblem:StateEquation-1}, we have the source term $f \in L^2(\Omega)$. The constant $\alpha > 0$ in \eqref{Equation:EllipticModelProblem:StateEquation-2} models the heat exchange with the outside of the domain $\Omega$, where a temperature of $y_a \in L^2(\Gamma)$ is assumed. In total, the parameter space is given by $\Us = \R^m \times \R$ and any parameter $u \in \Us$ can be written as the vector $u = (u^\kappa,u^r)^T$ with $u^\kappa = (u^\kappa_1,\ldots,u^\kappa_m)^T \in \R^m$. Setting $H=L^2(\Omega)$ and $V=H^1(\Omega)$ the weak formulation of \eqref{Equation:EllipticModelProblem:StateEquation} is
\begin{align}
a(u;y,\varphi) = \mathcal{F}(\varphi) \quad \mbox{ for all } \varphi \in V \label{Equation:EllipticModelProblem:WeakFormulation}
\end{align}
for any $u\in\Us$. In \eqref{Equation:EllipticModelProblem:WeakFormulation} the parameter-dependent symmetric bilinear form $a(u;\cdot\,,\cdot) \colon V \times V \to \R$ is given by
\begin{align*}
a(u;\varphi,\psi) := & \; \sum_{i=1}^m u_i^\kappa \int_{\Omega_i} \nabla \varphi(\bx) \cdot \nabla \psi(\bx) \dx + u^r \int_\Omega r(\bx) \varphi(\bx) \psi(\bx) \dx \\
& + \alpha \int_\Gamma \varphi(\bs) \psi(\bs) \ds
\end{align*}
for all $\varphi,\psi \in V$ and $u\in\Us$. The linear functional $\mathcal{F} \in V'$ is defined by
\[
\mathcal{F}(\varphi) := \int_\Omega f(\bx) \varphi(\bx) \dx + \alpha \int_\Gamma y_a(\bs) \varphi(\bs) \ds \quad \mbox{ for all } \varphi \in V.
\]

\begin{Lemma} \label{Lemma:EllipticModelProblem:PropertiesOfBilinearForm}
	\begin{enumerate}
		\item [\em a)] For all $u \in \Us$ it holds
		\begin{align*}
		\left\Vert a(u;\cdot,\cdot) \right\Vert_{L(V,V')} \leq C \left\Vert u \right\Vert_\Us
		\end{align*}
		with a constant $C > 0$, which does not depend on $u$.
		\item [\em b)] For all $u \in \Us$ with $u^\kappa > 0$ in $\mathbb R$ and $u^r > 0$, it holds
		\begin{align*}
		a(u;\varphi,\varphi) & \geq \min \left( u^\kappa_1,\ldots,u^\kappa_m , u^r \right) \left\Vert \varphi \right\Vert_V^2 \quad \mbox{for all } \varphi \in V.
		\end{align*}
		\item [\em c)] The mapping $\mathcal{F} \in V'$ is well-defined.
	\end{enumerate}
\end{Lemma}

\begin{proof}
	All statements follow from similar arguments of \cite[Lemma 1.4]{Mec19}, where related operators were considered in the parabolic case.\hfill\quad
\end{proof}

\begin{Theorem} \label{Theorem:EllipticModelProblem:UniqueSolvability}
	Let $u \in \Us$ with $u > 0$ be arbitrary. Then there is a unique solution $y = y(u) \in V$ of \eqref{Equation:EllipticModelProblem:StateEquation}. Moreover, the estimate
	\begin{align}
	\left\Vert y \right\Vert_V \leq C \left( \left\Vert f \right\Vert_{L^2(\Omega)} + \left\Vert y_a \right\Vert_{L^2(\Gamma)} \right) \label{Equation:EllipticModelProblem:AprioriEstimate}
	\end{align}
	holds with a constant $C > 0$, which depends continuously on $u$, but is independent of $f$ and $y_a$.
\end{Theorem}

\begin{proof}
	The claims follow from the Lax-Milgram theorem (cf. \cite{Eva10}) and Lemma~\ref{Lemma:EllipticModelProblem:PropertiesOfBilinearForm}.\hfill\quad
\end{proof}

\begin{Definition}
	Let $u^\kappa_{\textsl{min}} \in (0,\infty)^m$ and $u^r_{\textsl{min}} > 0$ be arbitrary. Then we define the closed set
	\[
	\Ueq := \{ u \in \Us \mid u^\kappa \ge u^\kappa_{\textsl{min}}, \; u^r \ge u^r_{\textsl{min}} \}.
	\]
	In view of {\rm Theorem~\ref{Theorem:EllipticModelProblem:UniqueSolvability}}, it is possible to define the solution operator $\mathcal{S} \colon \Ueq \to V$, which maps any parameter $u \in \Ueq$ to the unique solution $y = \mathcal{S}(u) \in V$ of \eqref{Equation:EllipticModelProblem:WeakFormulation}.
\end{Definition}

\begin{Remark} \label{Remark:EllipticModelProblem:ExistenceAlphaMin}
	Due to Lemma~\rm{\ref{Lemma:EllipticModelProblem:PropertiesOfBilinearForm}}, we can conclude that $a(u;\varphi,\varphi) \geq \alpha_{\mathsf{min}} \left\Vert \varphi \right\Vert_V^2$ for all $\varphi \in V$ and $u \in \Ueq$, where $\alpha_{\mathsf{min}} := \min \left( (u^\kappa_{\textsl{min}})_1,\ldots,(u^\kappa_{\textsl{min}})_m , u^r_i \right) > 0$. In particular, the constant $C$ in \eqref{Equation:EllipticModelProblem:AprioriEstimate} can be chosen independently of $u$ if we restrict ourselves to parameters $u \in \Ueq$. 
\end{Remark}

\begin{Theorem}
	The solution operator $\mathcal{S} \colon \Ueq \to V$ is twice continuously Fréchet differentiable. For the first derivative $\mathcal{S}' \colon \Ueq \to L(\Us,V)$, we have that for any $u \in \Ueq$ and $h \in \Us$ the function $y^h := \mathcal{S}'(u)h \in V$ solves the equation
	\begin{align*}
	a(u;y^h,\varphi) = - \partial_u a(u;\mathcal{S}(u),\varphi)h \quad \mbox{ for all } \varphi \in V.
	\end{align*}
	The second derivative $\mathcal{S}'' \colon \Ueq \to L(\Us,L(\Us,V))$ is given as follows: For any $u \in \Ueq$ and $h_1,h_2 \in \Us$, the function $y^{h_1,h_2} := \mathcal{S}''(u)(h_1,h_2)$ solves the equation
	\begin{align*}
	a(u;y^{h_1,h_2},\varphi) = - \partial_u a(u;\mathcal{S}'(u)h_1,\varphi)h_2 - \partial_u a(u;\mathcal{S}'(u)h_2,\varphi)h_1 \quad \mbox{ for all } \varphi \in V. 
	\end{align*}
\end{Theorem}
\begin{Remark} \label{Remark:EllipticModelProblem:PartialDerivativeOfBilinearForm}
	By $\partial_u a$ we denote the partial derivative of the mapping $a$ w.r.t.~the parameter $u$. Since $a$ is linear in $u$, it holds
	\begin{equation*}
	\begin{aligned}
	\partial_u a(u;\varphi,\psi) h = a(h;\varphi,\psi), && \partial^2_u a(u;\varphi,\psi) = 0 \in L(\Us,\Us')
	\end{aligned}
	\end{equation*}
	for all $u,h \in \Us$ and all $\varphi,\psi \in V$. In particular, we can identify $\partial_u a(u;\varphi,\psi) \in \Us'$ by
	\begin{align*}
	\partial_u a(u;\varphi,\psi) = \left( \begin{array}{c}
	\int_ {\Omega_1} \nabla \varphi(\bx)\cdot\nabla \psi(\bx) \dx \\
	\vdots \\
	\int_{\Omega_m} \nabla \varphi(\bx)\cdot\nabla \psi(\bx) \dx \\
	\int_\Omega r(\bx) \varphi(\bx) \psi(\bx) \dx
	\end{array} \right) \in \Us
	\end{align*}
	by using the Riesz representation theorem. 
\end{Remark}
We are now ready to state the multiobjective parametric PDE-constrained optimization problem (MPPOP). Let $k \in \N$ be fixed and 
\[
\sigma_\Omega^{(1)},\ldots,\sigma_\Omega^{(k)} \geq 0 \quad \mbox{ as well as } \quad \sigma_\Us^{(1)},\ldots,\sigma_\Us^{(k)} \geq 0 
\]
be non-negative weights. Furthermore, denote by $y_\Omega^{(1)},\ldots,y_\Omega^{(k)} \in H$ the desired states and by $u_d^{(1)},\ldots,u_d^{(k)} \in \Us$ the desired parameters. Then we define the multiobjective essential cost functions $\Jhat_1,\ldots,\Jhat_k \colon \Ueq \to \R$ by
\begin{align*}
\Jhat_i(u) := \frac{\sigma_\Omega^{(i)}}{2} \big\|\mathcal{S}(u) - y_\Omega^{(i)} \big\|_H^2 + \frac{\sigma_\Us^{(i)}}{2} \,\big\|u - u_d^{(i)} \big\|_\Us^2\quad\text{for all }u \in \Ueq\text{ and }i \in \{1,\ldots,k\}.
\end{align*}
Moreover, $u_a,u_b$ with $u_a \leq u_b$ are lower and upper bounds on the parameter $u$ which we assume to be finite. We define $\Uad:= \{ u \in \Us \mid u_a \leq u \leq u_b \}$ and we assume that $\Uad \subset \Ueq$ holds. Note that $\Uad$ is a closed, convex and bounded set because of the finiteness assumption on $u_a$ and $u_b$. We are interested in solving
\begin{align}
\min_{u \in \Uad} \Jhat(u) = \min_{u \in \Uad} \big(\Jhat_1 (u),\ldots,\Jhat_k (u)\big)^T. \tag{\textbf{MPPOP}} \label{Equation:EllipticModelProblem:MPOP}
\end{align}
Note that, thanks to the assumptions on $\Uad$ and $\sigma_\Us^{(i)}$, the costs $\Jhat_1,\ldots,\Jhat_k$ are upper semi-continuous and Assumption~\ref{Assumption:NonConvexMOP} is also satisfied. This problem fits into the framework of non-convex multiobjective optimization and Algorithm~\ref{Algorithm:PSM:HierarchicalAlgorithm:Numerical} can be applied. The non-convexity comes from the way the bilinear form depends on the parameter $u$. This makes, in fact, the solution mapping non-linear and thus the MPPOP non-convex. To close this section, we derive the expression of the gradient and Hessian of the cost functionals $\Jhat_1,\ldots,\Jhat_k$. We define the $i$-th adjoint equation and its solution operator as
\begin{Definition} \label{Definition:EllipticModelProblem:SolutionOperatorAdjointEquation}
	For $i=1,\ldots,k$, the solution operator of the $i$-th adjoint equation is $\mathcal{A}_i \colon \Ueq \to V$, where for any given $u \in \Ueq$, $p^{(i)} := \mathcal{A}_i(u)$ solves the equation
	\begin{equation}\label{Equation:EllipticModelProblem:AdjointEquation:WeakFormulation}
	\begin{aligned}
	a(u;\varphi,p^{(i)}) = \langle \sigma_\Omega^{(i)} (\mathcal{S}(u) - y_\Omega^{(i)}) , \varphi \rangle_H \quad \mbox{ for all } \varphi \in V.
	\end{aligned}
	\end{equation}
\end{Definition}
As shown in \cite{Ban21}, this operators satisfy the two following results:
\begin{Lemma} \label{Lemma:EllipticModelProblem:AdjointSolutionOperatorDifferentiable}
	The solution operator $\mathcal{A}_i \colon \Ueq \to V$ is continuously Fréchet differentiable for all $i=1,\ldots,k$. For all $i=1,...k$, for the first derivative ${\mathcal{A}}_i' \colon \Ueq \to L(\Us,V)$, we have that for any $u \in \Ueq$ and $h \in \Us$ the function $p^{(i),h} := \mathcal{A}_i'(u)h \in V$ solves the equation
	\begin{align}
	a(u;\varphi,p_i^{(i),h}) = - \partial_u a(u;\varphi,\mathcal{A}_i(u))h + \sigma_\Omega \langle \mathcal{S}'(u) h , \varphi \rangle_{V',V} \quad \mbox{ for all } \varphi \in V. \label{Equation:EllipticModelProblem:AdjointEquation:FirstDerivative:WeakFormulation}
	\end{align}
\end{Lemma}
\begin{Corollary} \label{Corollary:EllipticModelProblem:AdjointRepresentationHessian}
	Let $\Uad\subset \Ueq$, $u \in \Uad$ and $h\in\Us$ be arbitrary. Then for $i=1,\ldots,k$ the cost functions $\Jhat_i$ are twice continuously Fréchet differentiable and it holds
	\begin{align*}
	\nabla \Jhat_i(u) &= - \partial_u a(u;\mathcal{S}(u),\mathcal{A}_i(u)) + \sigma_\Us (u - u_d^{(i)}) \in \Us,\\
	\nabla^2 \Jhat_i(u) h &= - \partial_u a(u;\mathcal{S}'(u)h,\mathcal{A}_i(u)) - \partial_u a(u;\mathcal{S}(u),\mathcal{A}_i'(u)h) + \sigma^{(i)}_\Us h \in \Us.
	\end{align*}
	where we use the representation of $\partial_u a(u;\mathcal{S}(u),\mathcal{A}_i(u)) \in \Us'$ in $\Us$, cf.~{\rm Remark~\ref{Remark:EllipticModelProblem:PartialDerivativeOfBilinearForm}}.
\end{Corollary}

\subsection{The RB method for MPPOP}

One of the limitations of solving the MPPOP directly with the PSM is the high computational cost. Algorithm~\ref{Algorithm:PSM:HierarchicalAlgorithm:Numerical}, in fact, requires to solve the state and adjoint equation a large number of times in order to efficiently approximate the Pareto set. Unfortunately, the numerical evaluation of the state and adjoint solution operators is costly due to the high number of degrees of freedom required to apply, for example, the FE method. For this reason, we use the RB method. In the following we explain how the RB method can be applied to our model. From Theorem~\ref{Theorem:EllipticModelProblem:UniqueSolvability}, we know that the weak form of the state equation admits a unique solution for any control $u\in\Ueq$. This allows us to define the solution operator $\mathcal{S}:\Ueq\to V$. Now, let us consider the so-called solution manifold $\mathcal{M}:= \{\mathcal{S}(u)\,|\, u\in\Ueq\}$. The goal of the RB method is to provide a low-dimensional subspace $V^\ell\subset V$, which is a good approximation of $\mathcal M$. The subspace $V^\ell$ is defined as the span of linearly independent snapshots $\mathcal{S}(u_1),\ldots,\mathcal{S}(u_\ell)$ for selected parameters $u_1,\ldots,u_\ell\in \Ueq$. Clearly, $V^\ell$ has dimension $\ell$ and the snapshots constitute its basis. Let us postpone the discussion on how to select good parameters for generating $V^\ell$. Given an RB space $V^\ell$, we obtain the reduced-order state equation by a Galerkin projection:
\begin{align}
a(u;y^\ell,\psi) = \mathcal{F}(\psi) \quad \mbox{ for all } \psi \in V^\ell. \label{Equation:Appendix:MOR:RB:EllipticModelProblem:RBWeakFormulation}
\end{align}
Also for the reduced-order equation, we have unique solvability for all parameters $u \in \Ueq$. The solution map $\mathcal{S}^\ell \colon \Ueq \to V^\ell$, which maps any parameter $u \in \Ueq$ to the unique solution $y^\ell = \mathcal{S}^\ell(u) \in V^\ell$ of \eqref{Equation:Appendix:MOR:RB:EllipticModelProblem:RBWeakFormulation}, is then well-defined. We can similarly define a reduced-order adjoint equation and essential cost functional. For $i=1,\ldots,k$, we define the essential reduced-order cost functions $\Jhat^\ell_i \colon \Ueq \to \R$ by
\begin{align*}
\Jhat^\ell_i(u) := \frac{\sigma_\Omega^{(i)}}{2} \|\mathcal{S}^\ell(u) - y_\Omega^{(i)} \|_H^2 + \frac{\sigma_\Us^{(i)}}{2} \|u - u_d^{(i)}\|_\Us^2,
\end{align*} 
the reduced-order adjoint equation by
\begin{equation}\label{Equation:Appendix:MOR:RB:EllipticModelProblem:RBAdjointEquation:WeakFormulation}
\begin{aligned}
a(u;\psi,p^{(i),\ell}) = \big\langle \sigma_\Omega^{(i)} (\mathcal{S}^\ell(u) - y_\Omega^{(i)}) , \psi \big\rangle_H \quad \mbox{ for all } \psi \in V^\ell
\end{aligned}
\end{equation}
and the reduced-order adjoint solution operator $\mathcal{A}^\ell_i:\Ueq\to V$. Following Corollary~\ref{Corollary:EllipticModelProblem:AdjointRepresentationHessian}, it is possible to represent the gradient and the Hessian of the essential reduced-order cost functions $\Jhat_i^\ell$ for $i=1,\ldots,k$ by simply replacing the operators $\mathcal{S}$ and $\mathcal{A}_i$ by their respective reduced-order versions $\mathcal{S}^\ell$ and $\mathcal{A}^\ell_i$. There are still two aspects which remain to be clarified: first, how to generate an RB space which guarantees a good approximation of the state and adjoint solution manifolds and, second, how to estimate a-posteriori (i.e., without explicitly evaluating the full-order solution operators $\mathcal{S}$ and $\mathcal{A}$) the error of such an approximation. 

For the first aspect, one can think of building an RB space either prior to solving the reduced-order optimization problem or while solving it. The first approach is the so-called offline/online decomposition; cf. \cite{Haa17}. This technique exploits a greedy algorithm in the offline phase, which iteratively searches for the parameter for which the approximation error between the full- and reduced-order state and adjoint variables is the largest. Then, the RB space is enriched (by solving the full-order state and adjoint equations at the respective parameter and orthonormalizing the newly computed snapshots with respect to the current RB basis) until a pre-defined tolerance for the approximation error is reached. Once the RB space is computed, the online phase can start: the optimization problem is solved fast on the reduced-order level. Although this technique is still widely used in literature, it shows many disadvantages in the context of optimization. At first, it suffers from the curse of dimensionality: for a high-dimensional parameter space it is too costly to explore the entire parameter space with a greedy procedure. At second, it is counter-intuitive to prepare an RB space which is accurate enough for any parameter, when usually the optimization method follows a (short) pattern in the parameter space to find the solution or when the Pareto set is contained in some local regions of the parameter space, as often in the case of non-convex multiobjective problems. Luckily, the focus has shifted recently towards adapting the RB space while proceeding with the optimization method. This procedure is followed, e.g., by the methods presented in \cite{Yue13,Qia17,Kei21,Ban22}. Let us specify that in \cite{Qia17,Kei21,Ban22} the authors proposed and progressively improved an RB method combined with a TR algorithm, based on more general results presented in \cite{Yue13}. Such a method constructs the RB space adaptively while the optimizer is computing the optimal solution. Our focus here is on further improving the method in \cite{Ban22}, which can be considered the most general among the TR-RB methods.  

For any of the above-mentioned methods, a-posteriori error estimates are crucial to compute upper bounds of the approximation error made by the RB space in reconstructing the solution for a given parameter without any full-order solution at hand. In case of optimization, one is also interested in estimating the error in reconstructing the cost functional and its gradient. For our model, we can use the following estimates:

\begin{Theorem} \label{Theorem:Appendix:MOR:RB:AposterioriEstimate}
	Let $u \in \Uad$ be arbitrary and denote by $\alpha(u)$ the coercivity constant of the bilinear form $a(u;\cdot,\cdot)$. By {\rm Remark~\ref{Remark:EllipticModelProblem:ExistenceAlphaMin}}, it holds $\alpha(u) \geq \alpha_{\mathsf{min}} > 0$. Let the residual $r_{\mathsf{st}}(u;\cdot) \in V'$ be given by $r_{\mathsf{st}}(u;\varphi) := \mathcal{F}(\varphi) - a(u;\mathcal{S}^\ell(u),\varphi)$ for all $\varphi \in V$. Then it holds
	\begin{align}
	\big\|\mathcal{S}(u) - \mathcal{S}^\ell(u)\big\|_V \leq \Delta_{\mathsf{st}}(u) := \frac{\big\|r_{\mathsf{st}}(u;\cdot) \big\|_{V'}}{\alpha(u)}. \label{Equation:Appendix:MOR:RB:AposterioriEstimateStateEquation}
	\end{align}
	For $i=1,\ldots,k$ the residual $r^{(i)}_{\mathsf{adj}}(u;\cdot) \in V'$ of the adjoint equations is given by $r^{(i)}_{\mathsf{adj}}(u;\varphi) := \langle \sigma_\Omega^{(i)} (\mathcal{S}^\ell(u) - y_\Omega^{(i)}) , \varphi \rangle_H - a(u;\varphi,\mathcal{A}^\ell_i(u))$ for all $\varphi \in V$. Then it holds
	\begin{align*}
	\big\|\mathcal{A}_i(u) - \mathcal{A}^\ell_i(u) \big\|_V \leq \Delta^{(i)}_{\mathsf{adj}}(u) := \frac{\big\| r^{(i)}_{\mathsf{adj}}(u;\cdot)\big\|_{V'} + \sigma_\Omega^{(i)} \Delta_{\mathsf{st}}(u)}{\alpha(u)}. 
	\end{align*}
	Furthermore, for $i=1,\ldots,k$ we have
	\begin{align*}
	\big| \Jhat_i(u) - \Jhat_i^\ell(u) \big|&\leq \Delta_{\mathsf{st}}(u) \big\|r^{(i)}_{\mathsf{adj}}(u;\cdot) \big\|_{V'} + \sigma_\Omega^{(i)} \Delta_{\mathsf{st}}(u)^2 =: \Delta_{\Jhat_i^\ell}(u),\\
	\big\|\nabla \Jhat_i(u) - \nabla \Jhat_i^\ell(u) \big\|_\Us&\leq \left\Vert \partial_u a(u;\cdot,\cdot) \right\Vert\left(\big\|\mathcal{S}^\ell(u)\big\|_V \Delta^{(i)}_{\mathsf{adj}}(u) + \Delta_{\mathsf{st}}(u) \Delta^{(i)}_{\mathsf{adj}}(u) \right. \\
	& \hspace{26mm}\left.  + \Delta_{\mathsf{st}}(u) \big\|\mathcal{A}^\ell_i(u) \big\|_V \right) =: \Delta_{\nabla \Jhat_i^\ell}(u).
	\end{align*}
\end{Theorem}

\begin{proof}
	A proof of the a-posteriori error estimates for the state and adjoint can be found in \cite{Haa17}. For the cost function and the gradient, we refer to \cite[Proposition 2.5]{Kei21}.\hfill\quad
\end{proof}

Note that we only need the reduced-order state and adjoint state to evaluate the a-posteriori error estimates. For our example, the computation of the coercivity constant $\alpha(u)$ is cheap, see Lemma~\ref{Lemma:EllipticModelProblem:PropertiesOfBilinearForm}. In more general examples, this might not be the case. Thus, one often uses a quickly computable lower bound $\alpha_{\mathsf{LB}}(u)$ instead. Possible methods for computing such a lower bound are, e.g., the min-theta approach (cf. \cite{Haa17}) or the Successive Constraint Method (SCM) (cf. \cite{Roz08}). Note finally that the computation of the terms $\| r_{\mathsf{st}}(u;\cdot)\|_{V'}$ and $\| r^{(i)}_{\mathsf{adj}}(u;\cdot) \|_{V'}$ is not possible in an infinite-dimensional setting. Even after discretization with the FE method, the cost of computing such a term depends on the dimension of the full-order model, which contradicts the request of having a computationally cheap estimate. However, in our case, the parameter-separability of the bilinear form $a(u;\cdot\,,\cdot)$ can be exploited to preassemble certain quantities in such a way that the computational cost for evaluating $\|r_{\mathsf{st}}(u;\cdot) \|_{V'}$ and $\| r^{(i)}_{\mathsf{adj}}(u;\cdot) \|_{V'}$ only depends on the dimension of the RB space; see, e.g., \cite{Roz08}. Finally, we apply the RB method to \eqref{Equation:EllipticModelProblem:MPOP}: for a given RB space $V^\ell$ the reduced-order MPPOP reads
\begin{align}
\min \Jhat^\ell(u) =\big(\Jhat^\ell_1 (u),\ldots,\Jhat^\ell_k (u)\big)^T\quad\text{s.t.}\quad u\in\Uad.
\tag{$\textbf{MPPOP}^\ell$} \label{Equation:MOR:ReducedOrderMPOP}
\end{align}
For an arbitrary reference point $z \in \R^k$ and target direction $r\in\mathbb R^k$, the reduced-order PS problem reads
\begin{equation}
\min_{(u,t)}t\quad\text{s.t.}\quad(t,u)\in\mathbb R\times\Uad\text{ and }\Jhat^\ell_i(u) - z_i \leq t, \quad i = 1,\ldots,k.
\label{Equation:MOR:RB:ROMPascolettiSerafiniScalarization}
\tag{$\textbf{P}^{\mathsf{PS},\ell}_{z,r}$}
\end{equation}
One could then outline an algorithm similar to Algorithm~\ref{Algorithm:PSM:HierarchicalAlgorithm:Numerical} by using an offline/online splitting. Because of the above-mentioned disadvantages, we focus on combining the PSP+ with the TR-RB method from \cite{Ban22} and extend it with respect to the method in \cite{Ban21}. The TR method introduces new aspects to the RB implementation, such as the adaptive construction of the RB space; see next section for further details.

\section{The TR-RB method}
\label{Section:4}

We briefly introduce the method from \cite{Ban22} and clarify how to apply this in combination with the PSM. In Section~\ref{Section:TRRB:ReduceNumberBasis} we highlight our extension to this method and how this can reduce the computational time. The basic idea of a TR method is to compute a first-order critical point of a costly optimization problem by iteratively solving some cheap-to-solve approximations in local regions of the admissible space, where these model approximations can be trusted (i.e. are accurate enough). In such a way, one can derive a global method, which converges in a finite number of steps. For each outer iteration $j\geq 0$ of the TR method, the cheap approximation of the objective is generally indicated by $m^{(j)}$ and the trust regions are described by a radius $\delta^{(j)}$. To simplify the exposition, let us stick with the case $\Us = \R^m\times\R$, as in Section~\ref{Section:PDE:model}. The TR method solves then, for each $j\geq 0$, the following constrained optimization sub-problems
\begin{equation}
\label{Equation:TRsubproblemgeneral}
\min_{v\in \Us} m^{(j)}(v) \quad\text{s.t.}\quad \|v\|_2\leq \delta^{(j)}, \tilde{u}:= u^{(j)}+v\in\Uad.  
\end{equation}  
Under suitable assumptions, problem \eqref{Equation:TRsubproblemgeneral} admits a unique solution $\bar v^{(j)}$, which is used to compute the next outer iteration $u^{(j+1)}= u^{(j)}+\bar v^{(j)}$. To further simplify the presentation of the algorithm in \cite{Ban22}, let us present it for a general cost functional $\JJ$. Later in this section we will give more details about its application to the MPPOP and PSM. The TR-RB version of problem \eqref{Equation:TRsubproblemgeneral} is
\begin{equation}
\label{Equation:TRsubproblem}
\min_{\tilde{u}\in\Uad} \JJ^{\ell,(j)}(\tilde{u}) \quad\text{s.t.}\quad q^{(j)}:= \frac{\Delta_{\JJ^{\ell,(j)}}(\tilde{u})}{\JJ^{\ell,(j)}(\tilde{u})}\leq \delta^{(j)},
\end{equation}
where $\Delta_{\JJ^{\ell,(j)}}(\tilde{u})$ is an estimate for the error $|\JJ(\tilde{u})-\JJ^{\ell,(j)}(\tilde{u})|$. Looking at \eqref{Equation:TRsubproblem}, one clearly sees that the role of the model function $m^{(j)}$ is played by the reduced-order model cost functional. This is perfectly in line with the TR spirit of having a cheap-to-solve approximation of the original optimization problem. The trust regions are defined instead through the RB error estimator, which is in fact the way one should use to check the quality of the approximation. In \cite{Kei21} also the importance of introducing a correction term on the RB level is discussed to improve the performance of the method. We point out that this only has to be done if one chooses two separate RB spaces for state and adjoint equations (see also \cite{Ban22}). This will not be the case for our application.
In Algorithm~\ref{Algorithm:MOR:TRRB}, we report the method from \cite{Ban22}. In what follows, we guide the reader through the features of the algorithm. At first, we need to initizialize the reduced-order model at the initial guess $u^{(0)}$. This means computing $\mathcal{S}(u^{(0)})$ and $\mathcal{A}_i(u^{(0)})$ for $i=1,\ldots,k$ and generating the RB space $V^{\ell,(0)}$ as their span. Similarly, updating the RB space $V^{\ell,(j)}$ at the point $u^{(j+1)}$ means computing the full-order quantities $\mathcal{S}(u^{(j+1)})$ and $\mathcal{A}_i(u^{(j+1)})$ for $i=1,\ldots,k$ and adding them to the RB space by a Gram-Schmidt orthonormalization. 
\begin{algorithm}[btp]
	\caption{TR-RB algorithm
		\label{Algorithm:MOR:TRRB}}
	\begin{algorithmic}[1]
		\STATE Initialize the reduced-order model at $u^{(0)}$, set $j=0$ and \textsf{Loop\_flag}$=$\textsf{True}; 
		\WHILE{
			{\normalfont\textsf{Loop\_flag}}}
		\STATE Compute the AGC point $u^{(j)}_\mathsf{AGC}$ \label{ComputeAGC};
		\STATE Compute $u^{(j+1)}$ as solution of \eqref{Equation:TRsubproblem} with stopping criteria \eqref{Equation:MOR:TRRB:TerminationCond}\label{TRRB-optstep};
		\IF{\label{Suff_condition_TRRB}$\JJ^{\ell,(j)}(u^{(j+1)})+\Delta_{\JJ^{\ell,(j)}}(u^{(j+1)})<\JJ^{\ell,(j)}(u^{(j)}_\text{\rm{AGC}})$} 
		\STATE Accept $u^{(j+1)}$, set $\delta^{(j+1)}=\delta^{(j)}$, compute $\varrho^{(j)}$ and $g(u^{(j+1)})$\label{AcceptingIterateSufficientCondition};
		\IF { $g(u^{(j+1)}) \leq \tau_{\mathsf{FOC}}$ } 
		\STATE	Set \textsf{Loop\_flag}$=$\textsf{False};
		\ELSE
		\IF {$\varrho^{(j)}\geq\eta_\varrho$}  
		\STATE	Enlarge the TR radius $\delta^{(j+1)} = \beta_1^{-1}\delta^{(j)}$;	
		\ENDIF
		\IF {{\normalfont\textbf{not}} {\normalfont\textsf{Skip\_enrichment\_flag}$(j)$}}
		
		\STATE Update the RB model at $u^{(j+1)}$ 
		\label{UpdateRBModelSufficientCondition};
		\ENDIF
		\ENDIF
		
		\ELSIF  {\label{Nec_condition_TRRB}$\JJ^{\ell,(j)}(u^{(j+1)})-\Delta_{\JJ^{\ell,(j)}}(u^{(j+1)})> \JJ^{\ell,(j)}(u^{(j)}_\text{\rm{AGC}})$} 
		\IF { $\beta_1\delta^{(j)} \leq \delta_{\text{\rm{min}}}$ \textbf{ or } \normalfont\textsf{Skip\_enrichment\_flag}$(j-1)$ \label{forced_enrichment}} 
		\STATE Update the RB model at $u^{(j+1)}$; 
		\ENDIF
		\STATE Reject $u^{(j+1)}$, shrink the radius $\delta^{(j+1)} = \beta_1\delta^{(j)}$ and go to \ref{TRRB-optstep}; 
		
		\ELSE 
		\STATE Compute $\JJ(u^{(j+1)})$, $g(u^{(j+1)})$, 
		$\varrho^{(j)}$ 
		and set $\delta^{(j+1)} = \beta_1^{-1}\delta^{(j)}$; 
		\IF { $g(u^{(j+1)})\leq \tau_{\mathsf{FOC}}$ } 
		\STATE Set \textsf{Loop\_flag}$=$\textsf{False};
		\ELSE 
		\IF { \label{AcceptingIterateNoUpdateByExactComputation_condition}{\normalfont\textsf{Skip\_enrichment\_flag}$(j)$} \textbf{ and } $\varrho^{(j)}\geq \eta_\varrho$}
		\STATE Accept $u^{(j+1)}$\label{AcceptingIterateNoUpdateByExactComputation};
		\ELSIF {$\JJ(u^{(j+1)}) \leq \JJ^{\ell,(j)}(u^{(j)}_\text{\rm{AGC}})$} 
		\STATE Accept $u^{(j+1)}$ and update the RB model 
		\label{AcceptingIterateAndUpdateByExactComputation};	
		\IF {$\varrho^{(j)}<\eta_\varrho$} 
		\STATE Set $\delta^{(j+1)}=\delta^{(j)}$;
		\ENDIF	
		
		\ELSE
		\IF { $\beta_1\delta^{(j)} \leq \delta_{\text{\rm{min}}}$ \textbf{ or } \normalfont\textsf{Skip\_enrichment\_flag}$(j-1)$ \label{forced_enrichment_2}} 
		\STATE Update the RB model at $u^{(j+1)}$; 
		\ENDIF
		\STATE Reject $u^{(j+1)}$, set $\delta^{(j+1)} = \beta_1\delta^{(j)}$ and go to \ref{TRRB-optstep};		
		\ENDIF
		\ENDIF
		\ENDIF
		\STATE Set $j=j+1$;		
		\ENDWHILE
	\end{algorithmic}
\end{algorithm} 
In Line~\ref{ComputeAGC} of Algorithm~\ref{Algorithm:MOR:TRRB}, it is required to compute the so-called approximated generalized Cauchy (AGC) point. We report here its definition according to \cite{Yue13,Kei21}.
\begin{Definition} \label{Definition:MOR:TRRB:CauchyPoint}
	Let $\kappa \in (0,1)$ and $\kappa_{\mathsf{arm}} \in (0,1)$ be backtracking parameters. For the current iterate $u^{(j)}$ define $d^{(j)} := \nabla \mathcal{J}^{\ell,(j)}(u^{(j)})$. Let $\alpha \in \N$ be the smallest number for which the two conditions
	\begin{align}
	\mathcal{J}^{\ell,(j)} \big( P_{\Uad}(u^{(j)} - \kappa^\alpha d^{(j)}) \big) - \mathcal{J}^{\ell,(j)} (u^{(j)}) & \leq - \frac{\kappa_{\mathsf{arm}}}{\kappa^\alpha} \| P_{\Uad}(u^{(j)} - \kappa^\alpha d^{(j)}) - u^{(j)} \|_\Us^2, \label{CauchyPoint1} \\
	q^{(j)}(P_{\Uad}(u^{(j)} - \kappa^\alpha d^{(j)})) & \leq \delta^{(j)}
	\end{align}
	are satisfied, where $P_{\Uad} \colon \Us \to \Uad$ is the canonical projection onto the closed and convex set $\Uad$. Then we define the \textup{AGC point} as $u^{(j)}_{\mathsf{AGC}} := P_{\Uad}(u^{(j)} - \kappa^\alpha d^{(j)})$.
\end{Definition} 
The TR-RB subproblem \eqref{Equation:TRsubproblem} is then solved in Line~\ref{TRRB-optstep} using a projected Newton-CG algorithm with the AGC point as a warm start and the following termination criteria
\begin{equation}
\| u - P_{\Uad} ( u - \nabla \mathcal{J}^{\ell,(j)}(u) ) \|_\Us \leq \tau_{\mathsf{sub}}, \quad \beta_{\mathsf{bound}} \delta^{(j)} \leq q^{(j)}(u) \leq \delta^{(j)}. \label{Equation:MOR:TRRB:TerminationCond}
\end{equation}
The first condition in \eqref{Equation:MOR:TRRB:TerminationCond} is the standard first-order criticality condition with tolerance $\tau_{\mathsf{sub}}\in(0,1)$ and the second one was already introduced in \cite{Qia17} to avoid too many iterations close to the TR boundary, which is generally an area where we are already starting to trust the model function less. The parameter $\beta_{\mathsf{bound}}$ is usually chosen to be close to one exactly for this purpose.

An important aspect of TR methods is the decision to accept or reject the step $u^{(j+1)}$. Generally, one asks for the so-called sufficient decrease condition $\JJ^{\ell,(j+1)}(u^{(j+1)})\leq \JJ^{\ell,(j)}(u^{(j)}_\mathsf{AGC})$; cf. \cite{Yue13}. Note that this condition requires to update the RB space before being sure that the step will be accepted. If it is rejected, then we performed a costly update without the possibility of exploiting it. Because of this fact, \cite{Qia17} proposed a sufficient (Line~\ref{Suff_condition_TRRB}) and a necessary (Line~\ref{Nec_condition_TRRB}) condition for the sufficient decrease condition. In \cite{Kei21} it is also noted that the full-order quantities $\JJ(u^{(j+1)})$ and $\nabla\JJ(u^{(j+1)})$ are cheaply available after updating the RB space. Additionally, \cite{Ban22} introduced the possibility of skipping a redundant enrichment, which is particularly useful at the late stage of the method, where we are close to the optimum. This will prevent the dimension of the RB space from growing too fast, so that the cheap-to-solve property is preserved. The three conditions to be checked in order to decide whether to skip the update of the RB space are contained in the following skipping parameter
\begin{align*}
\normalfont\textsf{Skip\_enrichment\_flag} & (j) := \big( q^{(j)}(u^{(j+1)}) \leq \beta_q \delta^{(j+1)} \big) \quad \mathsf{and} \nonumber \\
& \hspace*{0.98cm} \left( \frac{\big| g(u^{(j+1)}) - g^{\ell,(j)}(u^{(j+1)}) \big|}{g^{\ell,(j)}(u^{(j+1)})} \leq \tau_g \right) \quad \mathsf{and} \nonumber \\
& \left( \frac{\big\| \nabla \mathcal{J}^{\ell,(j)}(u^{(j+1)}) - \nabla \mathcal{J}(u^{(j+1)}) \big\|_\Us}{\big\| \nabla \mathcal{J}^{\ell,(j)}(u^{(j+1)}) \big\|_\Us} \leq  \min \{ \tau_{\mathsf{grad}} , \beta_{\mathsf{grad}} \delta^{(j+1)} \} \right).
\end{align*}    
where $\beta_q,\beta_{\mathsf{grad}},\tau_g,\tau_{\mathsf{grad}}\in(0,1)$ are given parameters and
\begin{align*}
g(u) := \big\| u - P_{\Uad} ( u - \nabla \mathcal{J}(u) ) \big\|_\Us, \quad g^{\ell,(j)}(u) := \big\| u - P_{\Uad} ( u - \nabla \mathcal{J}^{\ell,(j)}(u) ) \big\|_\Us.
\end{align*}
Note also that $g(u)=0$ is nothing else than the standard first-order condition for optimization problems with constraints on the parameter set. This is the reason why Algorithm~\ref{Algorithm:MOR:TRRB} terminates when $g(u^{(j+1)})< \tau_{\mathsf{FOC}}$ holds with $\tau_{\mathsf{FOC}}\in(0,1)$. For more details on the skipping condition, we refer to \cite{Ban22}. Typically, TR methods also have the option of shrinking (enlarging) the TR radius $\delta^{(j)}$ with some factor $\beta_1\in(0,1)$ ($\beta_1^{-1}>1$, respectively). In the case of Algorithm~\ref{Algorithm:MOR:TRRB}, we shrink the radius if a point is rejected. We also compute the ratio
\begin{align*}
\varrho^{(j)} := \frac{\mathcal{J}(u^{(j)}) - \mathcal{J}(u^{(j+1)})}{\mathcal{J}^{\ell,(j)}(u^{(j)}) - \mathcal{J}^{\ell,(j)}(u^{(j+1)})}.
\end{align*}
If this ratio is greater than a parameter $\eta_\varrho\in[0.75,1]$, then the radius is enlarged. Algorithm~\ref{Algorithm:MOR:TRRB} is proved to be convergent given some technical assumptions on the problem. We summarize everything in the following theorem (cf. \cite{Ban22})
\begin{Theorem}
	\label{Theorem:TRRB:convergence}
	Suppose that $\Uad = [u^a,u^b] \subset \R^P$ for some $u^a,u^b \in \R^P$ with $u^a \leq u^b$. Assume that $\mathcal{J}$ and $\mathcal{J}^{\ell,(j)}$ ($j \in \N$) are strictly positive, $\mathcal{J}$ is continuously Fr\'echet differentiable and $\mathcal{J}^{\ell,(j)}$ is even twice continuously Fr\'echet differentiable for all $j \in \N$. Moreover, $\nabla \mathcal{J}^{\ell,(j)}$ is uniformly Lipschitz-continuous with respect to $j$. Suppose that there is $\delta_{\mathsf{min}} > 0$ such that for every $j \in \N$ there exists a TR radius $\delta^{(j)} \geq \delta_{\mathsf{min}}$, for which there is a solution $u^{(j+1)}$ of the TR-RB subproblem \eqref{Equation:TRsubproblem} which is accepted by {\rm Algorithm~\ref{Algorithm:MOR:TRRB}}. Assume that the family of functions $(q^{(j)})_{j \in \N}$ is uniformly continuous w.r.t. the parameter $u$ and the index $j$. Then every accumulation point $\ubar$ of the sequence of iterates $(u^{(j)})_{j \in \N}$ is a first-order critical point for the full-order optimization problem, i.e., it holds
	\[
	\left\Vert \ubar - P_{\Uad} \left( \ubar - \nabla \mathcal{J}(\ubar) \right) \right\Vert_\Us = 0.
	\]
	In particular, {\rm Algorithm~\ref{Algorithm:MOR:TRRB}} terminates after finitely many steps.	
\end{Theorem}
Although many of the assumptions in Theorem~\ref{Theorem:TRRB:convergence} are quite technical for the proof, one can show that they are reasonable in the case of the RB method; cf. \cite{Ban22}. 

\subsection{The TR-RB algorithm applied to the PS method}   

In this section we show how Algorithm~\ref{Algorithm:MOR:TRRB} can be applied to the PS method. To this end, we recall the following lemma from \cite{Ban21}.

\begin{Lemma} \label{Lemma:MOR:TRRB:CostFunctionsBoundedIndependentOfRBSpace}
	There are constants $C_J,C_{\nabla J},C_{\nabla^2 J} > 0$ such that for any $j \in \{1,\ldots,k\}$, any $u \in \Uad$ and any choice of the RB space $V^\ell$ it holds
	\begin{align*}
	\big| \Jhat_i^\ell(u) \big| \leq C_J, \quad \big\| \nabla \Jhat_i^\ell(u) \big\|_\Us \leq C_{\nabla J}, \quad
	\big\| \nabla^2 \Jhat_i^\ell(u) \big\|_{L(U)} \leq C_{\nabla^2 J}.
	\end{align*}
\end{Lemma}

Lemma~\ref{Lemma:MOR:TRRB:CostFunctionsBoundedIndependentOfRBSpace} immediately implies that the reduced-order gradient is uniformly Lipschitz-continuous with respect to $\ell$. We have to solve \eqref{Equation:PascolettiSerafiniScalarization}. We follow the approach in \cite{Ban21}, where the target direction $r = (1, . . . , 1)$ is chosen and an augmented Lagrangian method is used. Provided a penalty parameter $\mu>0$, the augmented Lagrangian for \eqref{Equation:PascolettiSerafiniScalarization} is
\begin{align}
\mathcal{L}_A((u,t,s),\lambda;\mu) := t + \sum_{i=1}^k \lambda_i c_i(u,t,s) + \frac{\mu}{2}\sum_{i=1}^k c_i(u,t,s)^2 \label{Equation:AugmentedLagrangian}
\end{align}
with $c_i(u,t,s) = \Jhat_i(u) - z_i - t+ s_i$. The idea is to iteratively solve the subproblems 
\begin{align}
\label{Equation:AugLagsubprob}
\min \mathcal{L}_A((u,t,s),\lambda;\mu) \quad \text{s.t.} \quad (u,t,s) \in \Uad \times \R \times \R_\geq^k 
\end{align}
approximately and then update the Lagrange multiplier $\lambda$ and the penalty parameter $\mu$ until the termination criteria
\begin{align}
\left\Vert c(u,t,s) \right\Vert_{\R^k} & < \tau_{\mathsf{EC}}, \label{Equation:Appendix:AugLagr:TerminationConditionEquConstr} \\
\big\| (u,t,s) - P_{\mathsf{ad}} \big( (u,t,s) - \nabla_{(u,t,s)} \mathcal{L}_A((u,t,s),\lambda;\mu) \big) \big\|_{\Us \times \R \times \R^k} & < \tau_{\mathsf{FOC}} \label{Equation:Appendix:AugLagr:TerminationConditionFOC}
\end{align}
are satisfied for some tolerances $\tau_{\mathsf{EC}}, \tau_{\mathsf{FOC}} \in (0,1)$, where $P_{\mathsf{ad}} \colon \Us \times \R \times \R^k \to \Uad \times \R \times \R_\geq^k$ is the canonical projection onto $\Uad \times \R \times \R_\geq^k$. For further details, we refer to \cite[Appendix B]{Ban21}. We want to combine then the augmented Lagrangian method with the TR-RB algorithm to solve problem \eqref{Equation:PascolettiSerafiniScalarization}. To do so, we apply Algorithm~\ref{Algorithm:MOR:TRRB} to solve each subproblem \eqref{Equation:AugLagsubprob}. We first define the reduced-order augmented Lagrangian 
\begin{align}
\mathcal{L}^{\ell}_A((u,t,s),\lambda;\mu) := t + \sum_{i=1}^k \lambda_i c_i^\ell(u,t,s)+ \frac{\mu}{2}\sum_{i=1}^k c_i^\ell(u,t,s)^2,
\end{align}
with $c_i^\ell(u,t,s)=\Jhat_i^\ell(u) - z_i - t+ s_i$, which leads to the reduced-order subproblem 
\begin{align}
\min \mathcal{L}^{\ell}_A((u,t,s),\lambda;\mu) \quad \text{s.t.} \quad (u,t,s) \in \Uad \times \R \times \R_\geq^k. \label{Equation:MOR:RB:PSM:AugmentedLagrangianSubproblem}
\end{align}
Note that in this case the admissible set $\Uad \times \R \times \R_\geq^k$ is unbounded, which collides with the first assumption of Theorem~\ref{Theorem:TRRB:convergence}. Nevertheless, \cite{Ban21} showed that the \eqref{Equation:PascolettiSerafiniScalarization} problem is also equivalent to
\begin{equation}
\min t \quad\text{s.t.}\quad (t,u) \in [t^{\mathsf{min}},t^{\mathsf{max}}]\times\Uad\text{ and }\Jhat(u) - z \leq t.
\label{Equation:MOR:TRRB:ModifiedPascolettiSerafiniScalarization}
\end{equation} 
There is still the problem that the admissible set for the slack variables $s$ is given by $[0,\infty)^k$. However, computing the partial derivative of the augmented Lagrangian $\mathcal{L}_A$ with respect to $s_i$, we obtain
\begin{align*}
\partial_{s_i} \mathcal{L}_A((u,t,s),\lambda;\mu) & = \lambda_i + \mu \left( \Jhat_i(u) - z_i - t + s_i \right)  \geq \lambda_i + \mu ( - z_i - t^{\mathsf{max}} + s_i).
\end{align*}
Thus, $\mathcal{L}_A$ is strictly monotonically increasing in $s_i$ for $s_i > - \lambda_i/\mu + z_i + t_{\mathsf{max}} =: s_i^{\mathsf{max}}$. Thus, given the Lagrange multiplier $\lambda$ and the penalty parameter $\mu$, we can restrict the slack variable $s_i$ to the interval $[0,s_i^{\mathsf{max}}]$. This will not cause any modification to the solvability and the solution of the augmented Lagrangian subproblem. By setting $\Xad := \Uad \times [t^{\mathsf{min}},t^{\mathsf{max}}] \times [0,s^{\mathsf{max}}]$, the equivalent formulation for the augmented Lagrangian subproblem corresponding to \eqref{Equation:MOR:TRRB:ModifiedPascolettiSerafiniScalarization} reads
\begin{align}
\min_{(u,t,s) \in \Xad} \mathcal{L}_A((u,t,s),\lambda;\mu). \label{Equation:MOR:TRRB:ModifiedLagrangianSubproblem}
\end{align}
Similarly, the reduced-order augmented Lagrangian subproblem is given by
\begin{align}
\min\mathcal{L}_A^\ell((u,t,s),\lambda;\mu)\quad\text{s.t.}\quad(u,t,s) \in \Xad. \label{Equation:MOR:TRRB:ModifiedRBLagrangianSubproblem}
\end{align}
Therefore, the goal is to apply Algorithm~\ref{Algorithm:MOR:TRRB} to solve the subproblem \eqref{Equation:MOR:TRRB:ModifiedLagrangianSubproblem}. To this end, we define $x=(u,t,s)\in\Us \times \R \times \R^k$, $\JJ(x) = \mathcal{L}_A(x,\lambda;\mu)$ and $\JJ^{\ell,(j)}(x) = \mathcal{L}_A^{\ell,(j)}(x,\lambda;\mu)$ for any reference point $z\in\R^k$, any Lagrange multiplier $\lambda\in\R^k_\geq$ and any penalty parameter $\mu>0$. Furthermore, using the a-posteriori estimates of the individual objectives (cf. Theorem~\ref{Theorem:Appendix:MOR:RB:AposterioriEstimate}), we have that
\begin{align*}
\big| \mathcal{J}(x) - \mathcal{J}^{\ell,(j)}(x) \big| \leq & \sum_{j=1}^k \left( \lambda_j + c \big| \Jhat^{\ell,(j)}_j(u) - z_j - t + s_j \big| \right) \Delta_{\Jhat_j^{\ell,(j)}}(u) \\
& + \sum_{j=1}^k \frac{c}{2} \left( \Delta_{\Jhat_j^{\ell,(j)}}(u) \right)^2 =: \Delta_{\mathcal{J}}^{\ell,(j)}(u)
\end{align*}
for all $u \in \Uad$, which can be used as a-posteriori error estimate in the TR-RB algorithm. According to Theorem~\ref{Theorem:TRRB:convergence}, we still need to show the strict positivity of the costs $\JJ$ and $\JJ^{\ell,(j)}$ and the uniform Lipschitz continuity of the gradient $\nabla \JJ^{\ell,(j)}$. For the first, we note that the objectives $\JJ$ and $\JJ^{\ell,(j)}$ are bounded from below by $C:= t^{\mathsf{min}} - \sum_{i=1}^k \lambda_i^2/(2\mu_i)$. Since $C$ depends only on fixed parameters of the optimization problems, we can add $C+1$ to the cost functions to obtain strict positivity. Obviously, this will not change the minimizers. The second property is a bit more technical and we prove it in the following lemma.
\begin{Lemma}
	Let the Lagrange multiplier $\lambda$ and the penalty parameter $\mu$ be given. Then the function $\JJ(\cdot):=\mathcal{L}_A(\cdot,\lambda;\mu)$ is twice continuously Fréchet-differentiable for all $j \in \N$ and the gradient $\nabla \mathcal{J}^{\ell,(j)}$ is uniformly Lipschitz continuous with respect to $j$.
\end{Lemma}

\begin{proof}
	Due to Corollary~\ref{Corollary:EllipticModelProblem:AdjointRepresentationHessian} the cost functions $\Jhat_1,\ldots,\Jhat_k$ are twice continuously Fréchet-differentiable. Thus, the function $(u,t,s) \mapsto \mathcal{L}_A((u,t,s),\lambda;\mu)$ is also twice continuously Fréchet-differentiable as a composition of twice continuously Fréchet-differentiable functions. Similarly, the reduced-order augmented Lagrangians $\mathcal{L}_A^{\ell,(j)}((\cdot\,,\cdot\,,\cdot),\lambda;\mu)$ are also twice continuously Fréchet-differentiable for all $j \in \N$. We have that
	\begin{align*}
	& \hspace*{3.5cm} \nabla^2 \mathcal{L}_A^{\ell,(j)}((u,t,s),\lambda;\mu) (h^u,h^t,h^s) = \\
	& \left( \begin{array}{c}
	\sum\limits_{j=1}^k \Big(\big( \lambda_j + \mu c_j^{\ell,(j)} \big) \nabla^2 \Jhat_j^{\ell,(j)}(u) h^u + \mu \big(d_j^{\ell,(j)} - h^t + h^s_j\big) \nabla \Jhat_j^{\ell,(j)}(u) \Big) \\
	k \mu h^t - \mu \sum\limits_{j=1}^k \big( d_j^{\ell,(j)} + h^s_j \big) \\
	\mu \big(d_1^{\ell,(j)} + h^s_1 - h^t \big) \\
	\vdots \\
	\mu \big(d_k^{\ell,(j)} + h^s_k - h^t \big)
	\end{array} \right)
	\end{align*}
	for any $h = (h^u,h^t,h^s) \in \Us \times \R \times \R^k$, where $c_j^{\ell,(j)} := \Jhat_j^{\ell,(j)}(u) - z_j - t + s_j$ and $d_j^{\ell,(j)} := \langle \nabla \Jhat^{\ell,(j)}(u),h^u \rangle_\Us$ for $j \in \{1,\ldots,k\}$. Using Lemma~\ref{Lemma:MOR:TRRB:CostFunctionsBoundedIndependentOfRBSpace}, we obtain that the Hessian matrix $\nabla^2 \mathcal{L}_A^{\ell,(j)}((u,t,s),\lambda;\mu)$ can be bounded independently of $(u,t,s)$ and $j$. Using the mean value theorem, we can conclude that the gradients $\nabla \mathcal{L}_A^{\ell,(j)}((\cdot,\cdot,\cdot),\lambda;\mu)$ are Lipschitz-continuous with constant $C_L$ uniformly in $j$.\hfill\quad
\end{proof}
As a consequence of Theorem~\ref{Theorem:TRRB:convergence}, we have that Algorithm~\ref{Algorithm:MOR:TRRB} applied to solve the augmented Lagrangian subproblem \eqref{Equation:MOR:TRRB:ModifiedLagrangianSubproblem} converges after finitely many steps to a first-order critical point of \eqref{Equation:MOR:TRRB:ModifiedLagrangianSubproblem}. 
\begin{Remark}
	\label{Remark:RBstrategies}
	Algorithm~{\em \ref{Algorithm:MOR:TRRB}} constructs and updates the RB space during the optimization procedure. In the case of the PS method, we are free to choose what to do for the space constructed during the TR-RB procedure. For example, we can use it for the next augmented Lagrangian subproblem (and also for the next reference point). We explored different ideas (see also \cite{Ban21}), but we report here only the two most interesting and efficient ones:
	\begin{enumerate}
		\item [\em 1)] Use one common RB space for all the subproblems and reference points, i.e. use a single space $V^\ell$ for solving the MOP. This strategy acquires efficiency in terms of reconstructing the full-order parameter space during the iteration. Therefore, thanks to the possibility of skipping an enrichment (which is the costly part in Algorithm~{\em\ref{Algorithm:MOR:TRRB}}), we expect more and more speed-up, together with accuracy, as the algorithm proceeds. 
		\item [\em 2)] Use multiple (local) RB spaces. This idea is already exploited by {\em\cite{Ban21,Bee18,Haa11}}. In this case, we do not use the previously obtained RB space for the next minimization problem. We generate instead $k$ initial spaces $V^\ell_1, \ldots, V^\ell_k$, resulting from the minimization\footnote{Note that this procedure does not require extra computational cost, since we need to solve these problems for the hierarchical PSM anyway} of the objectives $\Jhat_1,\dots,\Jhat_k$. Then at the beginning of every PS problem, we can decide to use the space $V^\ell_i$ for which $q^{(0)}(u^{(0)})<\beta_q\delta^{(0)}$ and $\dim V^\ell_i\leq \ell_\mathsf{max}$, with $\ell_\mathsf{max}\in\N$ being a predefined maximal number of basis functions. If several spaces satisfy these conditions then we select the one for which the value $q^{(0)}(u^{(0)})$ is the smallest. If instead there is no space fulfilling these conditions, we initialize a new space $V^\ell_{k+1}$ by using the full-order quantities $\mathcal{S}(u^{(0)})$ and $\mathcal{A}_i(u^{0})$ for $i=1,\ldots,k$. 
	\end{enumerate}
	Although these two techniques are already efficient, we noticed that there is a common problem: the number of RB basis functions might grow too fast and prevent a good speed-up for the solution. In particular, this is the case for the first strategy. To fix this issue, we propose different strategies to remove basis functions from $V^\ell$ in Section~{\em\ref{Section:TRRB:ReduceNumberBasis}}. This approach was not considered in {\em\cite{Ban21,Ban22,Kei21,Qia17}} and to our knowledge it has not been addressed in the literature yet.
	In reduced-order optimization, instead, this is meaningful, since the reduced-order model might grow too fast; see, e.g., {\em\cite{Mec19}}, in the case of proper orthogonal decomposition.
\end{Remark}

\subsection{How to reduce the number of basis functions}
\label{Section:TRRB:ReduceNumberBasis}

We point out that what is described in this section can also generally be applied to Algorithm~\ref{Algorithm:MOR:TRRB} from \cite{Ban22} without any relation to the PS method. Therefore, we use again the general notation $\JJ$ for the cost, as it was done in the beginning of this section. The methodology to remove a basis function comes from the observation that some basis elements might not be used during the optimization process. Suppose that we start from a point $u^{(0)}$ very far from the optimum. Clearly, after $j$ iterations the point $u^{(j)}$ is in a completely different region of the admissible set compared to the one of the starting point. Hence, the basis functions built for $u^{(0)}$ might give a negligible contribution in spanning the reduced-order model at the point $u^{(j)}$. If this is the case, we can expect that these functions will not play any further role also for the subsequent points and therefore they can be removed to reduce the dimension of the RB space. Our methodologies for removing basis functions are then based on Remark~\ref{Remark:RBstrategies} and try to check which basis functions give a negligible contribution for the current iteration of the TR-RB algorithm. Notice that every technique we propose from now on will be applied after updating the RB space in the TR-RB algorithm. The aim is to modify the updated RB space in order to provide a new RB space, where the number of basis functions is reduced.

\noindent
\textbf{Technique T1.} The first proposed technique is based on the computation of the so-called Fourier coefficients. Given $v\in V$ and a set of orthonormal basis functions $\{\psi_n\}_{n=1}^\ell\subset V^\ell$, the $n$-th \textup{Fourier coefficient} is defined as $c^{(n)}_\mathcal{F}(v):= \langle v,\psi_n\rangle_V$. Now, T1 consists in computing $c^{(n)}_\mathcal{F}(\mathcal{S}(u^{(j+1)}))$ and $c^{(n)}_\mathcal{F}(\mathcal{A}_i(u^{(j+1)}))$, $i=1,\ldots,k$, for $n=1,\ldots,\ell$ and remove the basis function $\psi_n$ for which 
\[
\zeta^{(n)}:=\max\left\{\frac{c_\mathcal{F}^{(n)}(\mathcal{S}(u^{(j+1)}))^2}{\sum_{\eta=1}^{\ell} c_\mathcal{F}^{(\eta)}(\mathcal{S}(u^{(j+1)}))^2},\max_{i=1,\ldots,k}\left\{\frac{c_\mathcal{F}^{(n)}(\mathcal{A}_i(u^{(j+1)}))^2}{\sum_{\eta=1}^{\ell} c_\mathcal{F}^{(\eta)}(\mathcal{A}_i(u^{(j+1)}))^2}\right\}\right\}
\]
is below a certain tolerance. Note, in fact, that the Fourier coefficients indicate the order of magnitude of the contribution of a given basis function in reconstructing the new snapshots that we want to add to update the RB. Strategy T1 is also based on the assumption that the snapshots, which we want to include in an update, are the most relevant for the new TR subproblem, because they correspond to the last accepted optimization step $u^{(j+1)}$. The advantage of T1 is that the required Fourier coefficients are already available from the Gram-Schmidt orthogonalization performed during the update of the RB space. There is, anyway, a possible drawback of T1 due to the tolerance we set: it can happen that also important basis functions are removed although one thinks that the tolerance is small enough. Because of this, we would like to have a criteria to decide in an unbiased way which basis functions should be removed.

\noindent
\textbf{Technique T2.} This approach is based on the idea that once a point $u^{(j+1)}$ is accepted by the TR-RB algorithm and the RB space is updated, we will compute a provisional AGC point $u^{(j+1),\mathsf{prov}}_{\mathsf{AGC}}$ (cf. Definition~\ref{Definition:MOR:TRRB:CauchyPoint}) with respect to the previously updated RB space. One robustness criteria that we demand is that after removing basis functions, this provisional AGC point is still inside the new TR \footnote{Note that the TR depends on the reduced-order model due to the inequality constraint in \eqref{Equation:TRsubproblem} and, therefore, changes if we remove basis functions.}, although it might not coincide with the actual AGC point $u^{(j+1)}_{\mathsf{AGC}}$ that we compute after removing basis functions according to Line~\ref{ComputeAGC} in Algorithm~\ref{Algorithm:MOR:TRRB} \footnote{Note that the reduced-order cost function changes by removing a basis function, so that also the first term in \eqref{CauchyPoint1} differs after this removal.}. If we do not demand this robustness criteria, we can expect a deterioration of the TR performances due to lack of accuracy of the RB model in the steepest descent direction. Another important aspect is to guarantee the convergence of the TR-RB method, which implies checking that the conditions for accepting the point $u^{(j+1)}$ are still fulfilled, although we removed basis functions. 

In summary, the difference with respect to T1 is then to remove basis functions starting from the one with the smallest value of $\zeta^{(n)}$ and proceeding in ascending order until one of the following conditions is satisfied
\begin{subequations}
	\label{T2Conditions}
	\begin{align}
	\frac{\Delta_{\JJ^{\ell-\mathsf{rem},(j+1)}}(u^{(j+1),\mathsf{prov}}_{\mathsf{AGC}})}{\JJ^{\ell-\mathsf{rem},(j+1)}(u^{(j+1),\mathsf{prov}}_{\mathsf{AGC}})} &>  \beta_q\delta^{(j+1)} ,\label{T2RmvBasa} \\
	\frac{\Delta_{\nabla\JJ^{\ell-\mathsf{rem},(j+1)}}(u^{(j+1),\mathsf{prov}}_{\mathsf{AGC}})}{\big\|\nabla\JJ^{\ell-\mathsf{rem},(j+1)}(u^{(j+1),\mathsf{prov}}_{\mathsf{AGC}})\big\|_\Us} &>  \min \{ \tau_{\mathsf{grad}} , \beta_{\mathsf{grad}} \delta^{(j+1)} \}, \label{T2RmvBasb}\\
	\frac{\big\| \nabla \mathcal{J}^{\ell-\mathsf{rem},(j+1)}(u^{(j+1)}) - \nabla \mathcal{J}(u^{(j+1)}) \big\|_\Us}{\big\| \nabla \mathcal{J}^{\ell-\mathsf{rem},(j+1)}(u^{(j+1)}) \big\|_\Us} &>  \min \{ \tau_{\mathsf{grad}} , \beta_{\mathsf{grad}} \delta^{(j+1)} \}, \label{T2RmvBasc} \\
	\frac{\big| g(u^{(j+1)}) - g^{\ell-\mathsf{rem},(j+1)}(u^{(j+1)}) \big|}{g^{\ell-\mathsf{rem},(j+1)}(u^{(j+1)})} &> \tau_g, \label{T2RmvBasd} \\
	\JJ^{\ell-\mathsf{rem},(j+1)}(u^{(j+1)}) &> \JJ^{\ell,(j)}(u^{(j)}_\mathsf{AGC}) ,\label{T2RmvBase}\\
	\mathcal{J}^{\ell-\mathsf{rem},(j+1)} \big( u^{(j+1),\mathsf{prov}}_{\mathsf{AGC}} \big) - \mathcal{J} (u^{(j+1)}) &> - \kappa_{\mathsf{arm}} \big\| u^{(j+1),\mathsf{prov}}_{\mathsf{AGC}} - u^{(j+1)} \big\|_\Us^2. \label{T2RmvBasf}
	\end{align}
\end{subequations}
If one of the conditions \eqref{T2Conditions} holds we re-add the basis function to the RB space and finish the removal continuing with the TR-RB procedure. T2 is summarized in Algorithm~\ref{Algorithm:T2}.

\begin{algorithm}
	\caption{Summary of T2
		\label{Algorithm:T2}}
	\begin{algorithmic}[1]
		\STATE Follow the steps in Algorithm~\ref{Algorithm:MOR:TRRB} until the RB model is updated at $u^{(j+1)}$;
		\STATE  \label{T2:firststep} Compute a provisional AGC point $u^{(j+1),\mathsf{prov}}_{\mathsf{AGC}}$ by using the reduced-order cost function w.r.t. the updated RB model;
		\STATE Compute $\zeta^{(n)}$ for $n \in \{1,\ldots,\ell\}$;
		\WHILE {None of the conditions in \eqref{T2Conditions} is fullfiled}
		\STATE Out of all remaining basis functions, remove the one with the smallest value of $\zeta^{(n)}$ from the RB space;
		\ENDWHILE 
		\STATE \label{T2:laststep} Add the last removed basis function to the RB space;
		\STATE Proceed with Algorithm~\ref{Algorithm:MOR:TRRB} with the RB space obtained performing Steps~\ref{T2:firststep}-\ref{T2:laststep};
	\end{algorithmic}
\end{algorithm}

Let us explain the meaning of \eqref{T2Conditions}. At first, the superindex $\ell-\mathsf{rem}$ indicates that the space used to compute the quantity is the RB space obtained after removing a basis function. Condition \eqref{T2RmvBasa} is to check that the provisional AGC point will remain inside in an accurate-enough region of the TR. Condition \eqref{T2RmvBasb} is in the spirit of \eqref{T2RmvBasa} but for the gradient of the objective. Conditions~\eqref{T2RmvBasc}-\eqref{T2RmvBasd} are based on the skipping enrichment criteria and are checked to ensure convergence and robustness of the method after the removal. For a similar issue we need to check that the sufficient decrease condition is fulfilled as well (cf. \eqref{T2RmvBase}). Finally, \eqref{T2RmvBasf} is to enforce that the provisional AGC point is still a Cauchy point. In such a way, we are sure that Algorithm~\ref{Algorithm:MOR:TRRB} converges even after performing the basis removal (cf. \cite{Kei21,Ban22}). In this sense, T2 introduces an unbiased way to deal with the technique introduced in T1. There are still a few aspects one should comment on before implementing T2. At first, note that all the above-mentioned conditions are cheaply computable, since they are based either on reduced-order quantities or the appearing full-order quantities are available because of the RB update. At second, conditions \eqref{T2RmvBasa} and \eqref{T2RmvBasb} request efficient and reliable error estimators. Although for the PSM the efficiency of $\Delta_\JJ^{\ell,(j)}$ is acceptable, it is not the same for an error estimator $\Delta_{\nabla\JJ}^{\ell,(j)}$ based on the a-posteriori estimates of the gradients of the individual objectives. These estimators generally produce a huge overestimation, which makes them useless in practice. We notice, in fact, that condition \eqref{T2RmvBasb} is immediately triggered in the case of the PSM and we can not remove any basis function. This is the reason why we solved this issue by two different related approaches:

\noindent
\textbf{Technique T2a.} We replace the numerator of \eqref{T2RmvBasb} by
\[
\big\| \nabla \mathcal{J}^{\ell-\mathsf{rem},(j)}(u^{(j+1),\mathsf{prov}}_{\mathsf{AGC}}) - \nabla \mathcal{J}(u^{(j+1),\mathsf{prov}}_{\mathsf{AGC}}) \big\|_\Us,
\]
which is the true error we wanted to estimate, but it is unfortunately costly. It requires the computation of the full-order quantities $\mathcal{S}(u^{(j+1),\mathsf{prov}}_{\mathsf{AGC}})$ and $\mathcal{A}_i(u^{(j+1),\mathsf{prov}}_{\mathsf{AGC}})$, $i=1,\ldots,k$.

\noindent
\textbf{Technique T2b.} We replace the numerator of \eqref{T2RmvBasb} by
\[
\big\|\nabla \mathcal{J}^{\ell-\mathsf{rem},(j)}(u^{(j+1),\mathsf{prov}}_{\mathsf{AGC}}) - \nabla \mathcal{J}^{\ell,(j+1)}(u^{(j+1),\mathsf{prov}}_{\mathsf{AGC}}) \big\|_\Us
\]
which is a cheap approximation of the true error that we suppose to be reliable only after enough steps of Algorithm~\ref{Algorithm:MOR:TRRB}, however.

\noindent
Clearly, if one has a good estimation of the gradient at hand, T2 can be still used in its original form.

\noindent
\textbf{Technique T3.} Another drawback of T2 is the fact that we first need to remove the basis function in order to check \eqref{T2Conditions}. This implies that when we stop the removal, we need to add back the last basis function which was removed, because it is containing important information; cf. Line~\ref{T2:laststep} of Algorithm~\ref{Algorithm:T2}. This results in a waste of time for the modified Algorithm~\ref{Algorithm:MOR:TRRB}. We decide to add the option of introducing numerical tolerances for each of the conditions \eqref{T2Conditions}. In such a way, the modified algorithm will generally stop before an important basis function is removed at the price of possibly leaving one or a few redundant basis functions in the RB space. We think that this is a meaningful modification regarding the time that is wasted reintroducing the removed basis function into the RB space; cf. Section~\ref{Section:numerics}. We indicate this last strategy as T3.

\section{Numerical experiments}
\label{Section:numerics}

In this section we test Algorithm~\ref{Algorithm:MOR:TRRB} and compare it with the results obtained in \cite[Section~3.2.2]{Ban21}. We use the same numerical setting, which we briefly report here. Let the domain $\Omega$ be the two-dimensional unit square, split into four different subdomains $\Omega_1 = (0,0.5) \times (0,0.5)$, $\Omega_2 = (0,0.5) \times (0.5,1)$, $\Omega_3 = (0.5,1) \times (0,0.5)$ and $\Omega_4 = (0.5,1) \times (0.5,1)$. For each $\Omega_i$, we consider a corresponding diffusion coefficient $u_i^{\kappa}\in\mathbb R$ in \eqref{Equation:EllipticModelProblem:StateEquation} for $i=1,\ldots,4$. The reaction term $r(x)$ is set to be constantly equal to $1$ for any $x\in\Omega$. We impose homogeneous Neumann boundary conditions (i.e., $\alpha=0$) and a source term $f(x)= \sum_{i=1}^4  c_i \chi_{\Omega_i}(x)$ with $c_1 \approx 2.76 $, $c_2 \approx -0.96 $, $c_3 \approx 0.51 $ and $c_4 \approx -1.66 $ generated randomly in order to obtain a problem with a non-convex Pareto front. For the spatial discretization of the state equation, we apply the Finite Element (FE) method with 1340 nodes and piecewise linear basis functions. For \eqref{Equation:EllipticModelProblem:MPOP} we choose the following three objectives
\begin{align*}
\Jhat_1(u) & := \frac{1}{2}\,\big\|\mathcal{S}(u) - y_\Omega^{(1)}\big\|_H^2 + \frac{\varepsilon}{2}\,\big\|u - u_d^{(1)}\big\|_\Us^2, \\
\Jhat_2(u) & := \frac{1}{2}\,\big\|\mathcal{S}(u) - y_\Omega^{(2)}\big\|_H^2 + \frac{\varepsilon}{2}\,\big\|u - u_d^{(2)}\big\|_\Us^2,\quad \Jhat_3(u)  := \frac{0.05}{2}\,\big\| u - u_d^{(3)}\big\|_\Us^2
\end{align*} 
with $\varepsilon=0.002$, the desired states
\begin{align*}
y_\Omega^{(1)}(x) := \chi_{(0,0.5) \times (0,1)}(x), \quad y_\Omega^{(2)}(x) := \chi_{(0.5,1) \times (0,1)}(x),
\end{align*}
and the desired parameter values
\begin{align*}
u_d^{(1)} = u_d^{(2)}  := (2,0,0,0,0.3)^T, \quad u_d^{(3)} & := (2,1,1,1,0.3)^T.
\end{align*}
The lower and upper parameter bounds are given by 
\[
u_a = (2,0.1,0.1,0.1,0.3)^T \quad \text{ and } \quad u_b = (2,4,4,4,0.3)^T,
\] 
respectively. This implies that $u_1^{\kappa} = 2$ and $u^r = 0.3$ are seen as constants and we only optimize over the three parameters $u_2^{\kappa}$, $u_3^{\kappa}$ and $u_4^{\kappa}$. Note furthermore, that the desired parameters $u_d^{(1)} = u_d^{(2)}$ are not admissible. In fact, as for the parameters of the source term, they were chosen such that the resulting Pareto front is non-convex. 
\begin{figure}[H]
	\centering
	\subfigure[]{\includegraphics[width=6.5cm]{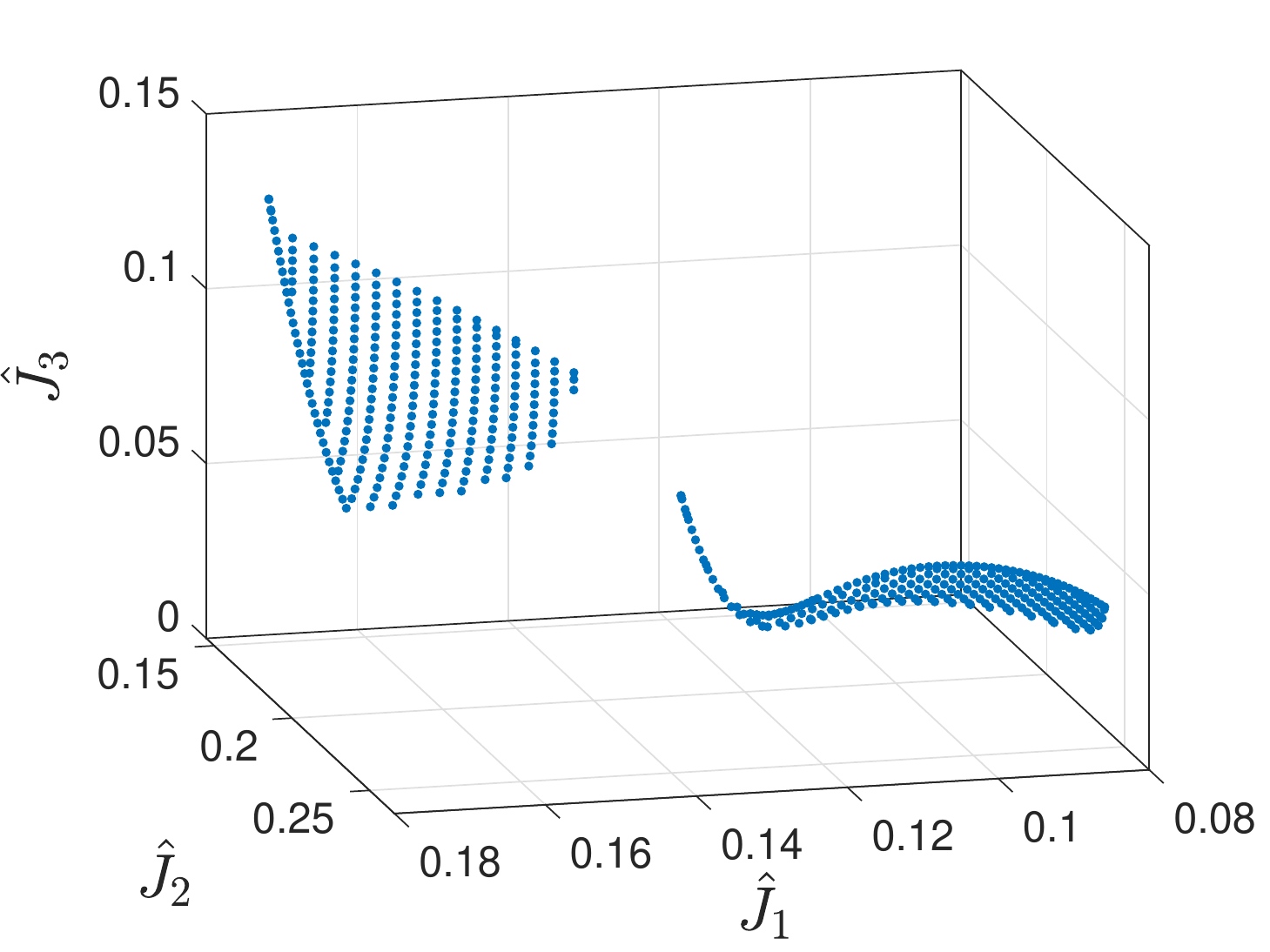}}
	\hspace{1mm}
	\subfigure[]{\includegraphics[width=6.5cm]{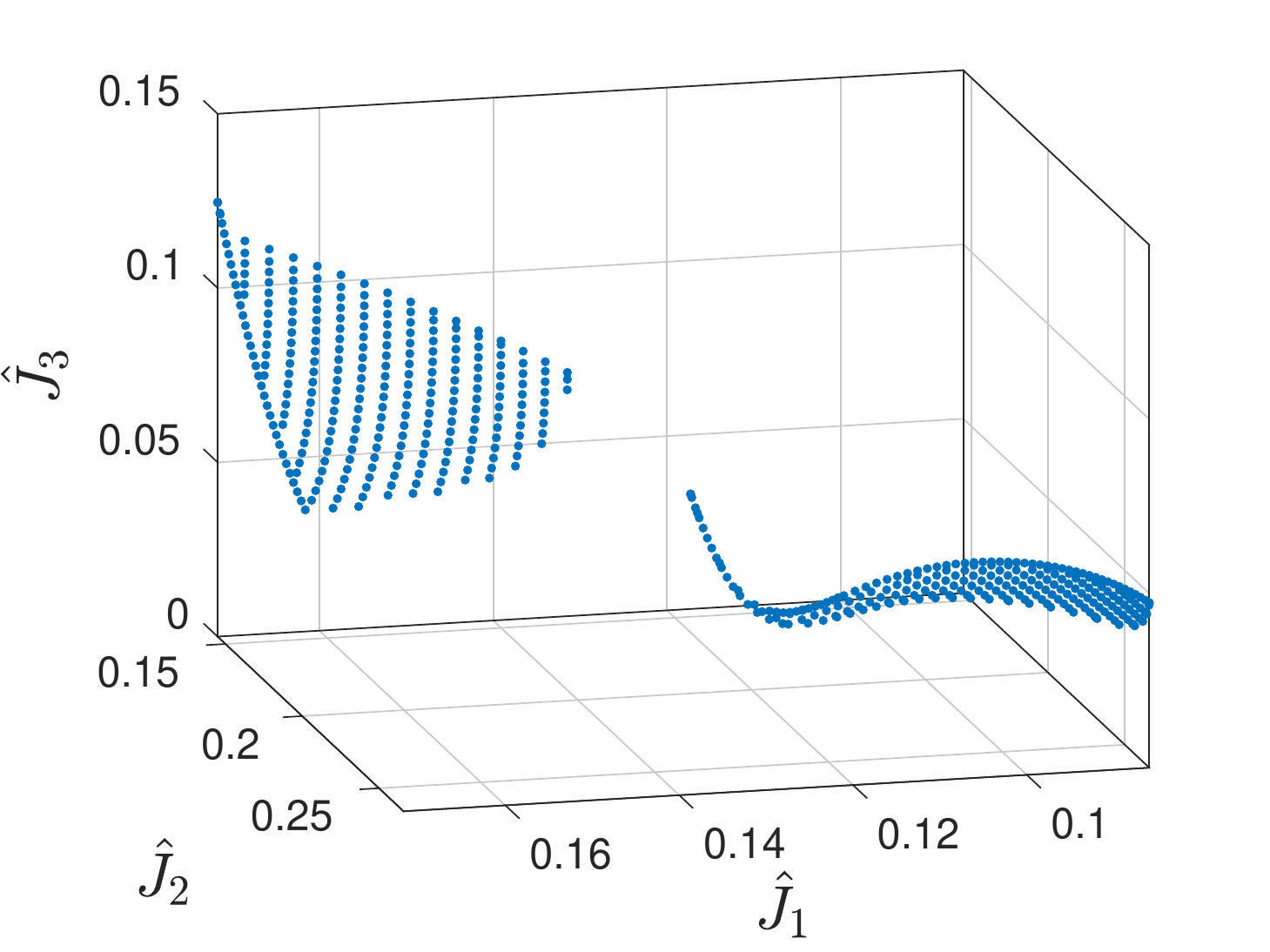}}
	\caption{(\textbf{a}) Algorithm~\ref{Algorithm:MOR:TRRB} no Removal local RB spaces. (\textbf{b}) Algorithm~\ref{Algorithm:MOR:TRRB} T3 local RB spaces. \label{Fig:ParetoFronts}}
\end{figure}
For the choice of the initial value for PSPs corresponding to reference points for the entire problem $(\Jhat_1,\Jhat_2,\Jhat_3)$ we do the following: Let $\ubar^i$ be the minimizer of $\Jhat_i$ for $i = 1,2,3$. Recall that the sets $D_i$ have been introduced in Definition~\ref{Definition:yidealpoint}-(ii). Then, if $z \in D_i$, we choose $\ubar^i$ as the initial value for solving \eqref{Equation:PascolettiSerafiniScalarization}. We additionally choose the shifting vectors $\tilde{d} = 0.001 \cdot (1,1,1)^T$, while the grid size $h$ for the reference point grid is set to $h_\textsl{PSM} = 0.003$. For detailed comments and results on the PSM applied on the FE and RB level, we refer to \cite[Section~3.2.2]{Ban21}. We report here only the necessary ones on RB level for a comparison with our proposed technique. Before doing that, let us mention that the tolerance chosen in T1 (cf. Section~\ref{Section:TRRB:ReduceNumberBasis}) for the Fourier coefficient is $10^{-6}$. Similarly, we choose the same tolerance for T3 in order to break the removal algorithm before deleting important basis functions, i.e. we subtract it on the right-hand side of \eqref{T2RmvBasa}-\eqref{T2RmvBasf}. At first, to validate our approach, we show in Figure~\ref{Fig:ParetoFronts} the obtained Pareto fronts by using the method in \cite{Ban21} (left) and our method (right). As one can see, there is no visible difference. The approximation error is, in fact, of the order of $10^{-6}$ for a Pareto point computed by all the proposed techniques (i.e., T1, T2a, T2b and T3) on average. 
\begin{figure}[H]
	\centering
	\includegraphics[width=7.5 cm]{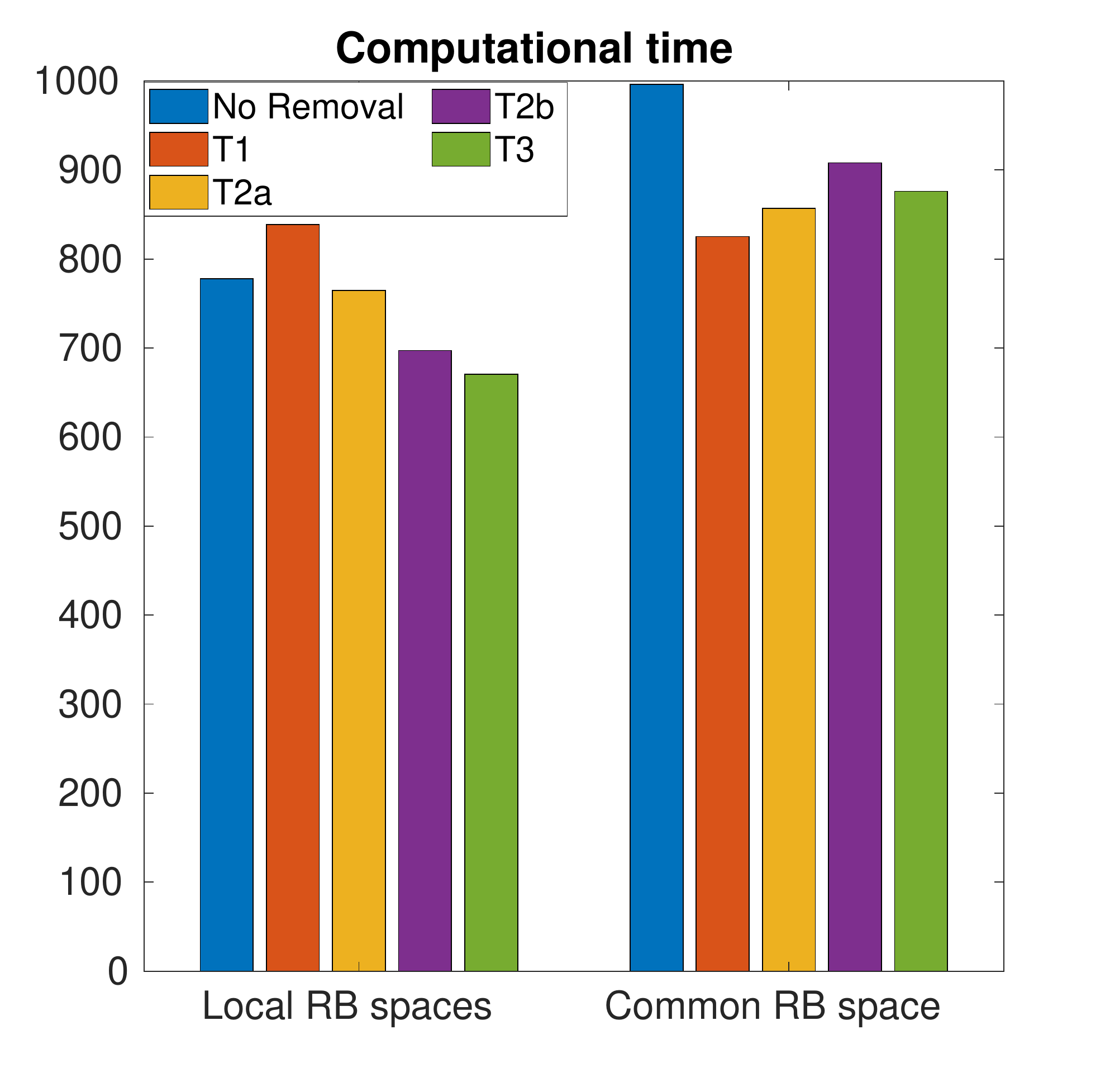}
	\caption{Computational times in seconds for Algorithm~\ref{Algorithm:MOR:TRRB} with or without basis removal and using the two strategies in Remark~\ref{Remark:RBstrategies} for initializing the RB space. \label{Fig:Times}}
\end{figure}
In Figure~\ref{Fig:Times} we compare the computational time of Algorithm~\ref{Algorithm:MOR:TRRB} with all the proposed techniques to the one of the algorithm in \cite{Ban21}. As one can see, we get a speed-up by using the proposed techniques in almost all cases. Depending on the strategy from Remark~\ref{Remark:RBstrategies}, one technique performs better than the others. Here we try to explain this phenomena in detail. Let us focus on the common RB space first. In this case, every technique helps in saving computational time. This is clearly the effect of removing redundant basis functions, which are particularly frequently included using a large common RB space. This is the reason why T1 appears to be the most effective, since it is the cheapest among the techniques (as we said it does not imply additional cost to be checked). T2a is more robust, but it comes with the price of evaluating the full-order gradient at the new AGC point and thus results to be slower than T1. Apparently, T2b should overcome this problem, but the inaccuracy of the RB space in the beginning give a bad approximation of \eqref{T2RmvBasb}, resulting in removing too many basis functions which leads to a worse approximation for the consecutive steps. This worsening of the approximation results in a way larger number of enrichment steps towards the end of the algorithm, which also negatively influences the computational time. T3 is comparable with T2a, meaning that for this example we are removing many basis functions in only a few instances, rather than frequently removing a few basis functions. Figure~\ref{Fig:Numberofbasis}(b) confirms the above remarks for the case of a common RB space. In this figure we report the number of basis functions obtained at the end of Algorithm~\ref{Algorithm:MOR:TRRB} while this is applied to compute each Pareto optimal point in the PS method.

Now, let us focus on the left group of columns in Figure~\ref{Fig:Times} (and thus on Figure~\ref{Fig:Numberofbasis}(a)), which corresponds to the computational times in the case of using local RB spaces (cf. Remark~\ref{Remark:RBstrategies}). This case is a bit more delicate, since the use of local RB spaces makes it more difficult to interpret the results. Here the problem of T1 is emerging. The fact that this technique removes a number of basis functions without any robustness criteria implies that the method slows down. In the case of local spaces, in fact, we do not have the same amount of redundant basis functions as it can occur for a common RB space. Therefore, we should only remove the basis functions which are actually redundant. As one can note in Figure~\ref{Fig:Numberofbasis}(a), T1 removes a significantly larger amount of basis functions in comparison to the other techniques. Here the criteria introduced in T2a play their role in a positive way. We can counteract the effect of T1 in such a way that the computational time is comparable to the one in \cite{Ban21}. The further simplification introduced in T2b helps to get an additional speed-up. In contrast to the common RB space, here we have local spaces which provide a sufficiently good accuracy for approximating \eqref{T2RmvBasb} also in the beginning of the optimization. This is then beneficial for the algorithm, since the cost of computing the criteria in T2b is way cheaper than T2a, where we need full-order solves of the state and adjoint equation to compute the gradient at the new AGC point. Additionally, T3 further improves T2a and T2b in terms of computational time, since in the case of local RB spaces it is more probable that we indeed remove only a few basis functions but more frequently than in the case of one common RB space. In this case, it is important to have tolerances that let us stop before removing an important basis function and save time for reintroducing it in the RB space. 
\begin{figure}
	\centering 
	\subfigure[]{\includegraphics[width=6.8cm]{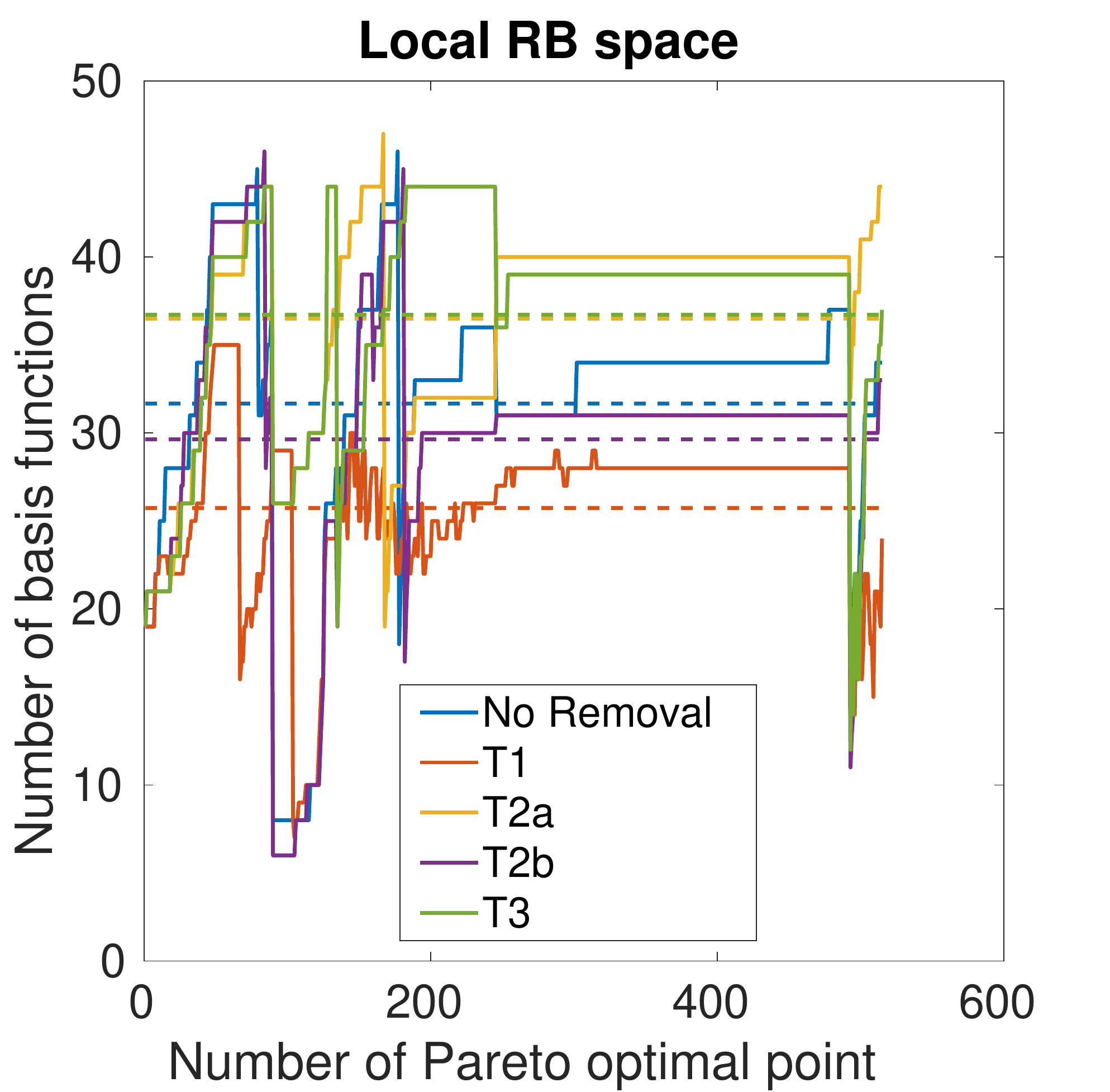}}
	\hspace{1mm}
	\subfigure[]{\includegraphics[width=6.8cm]{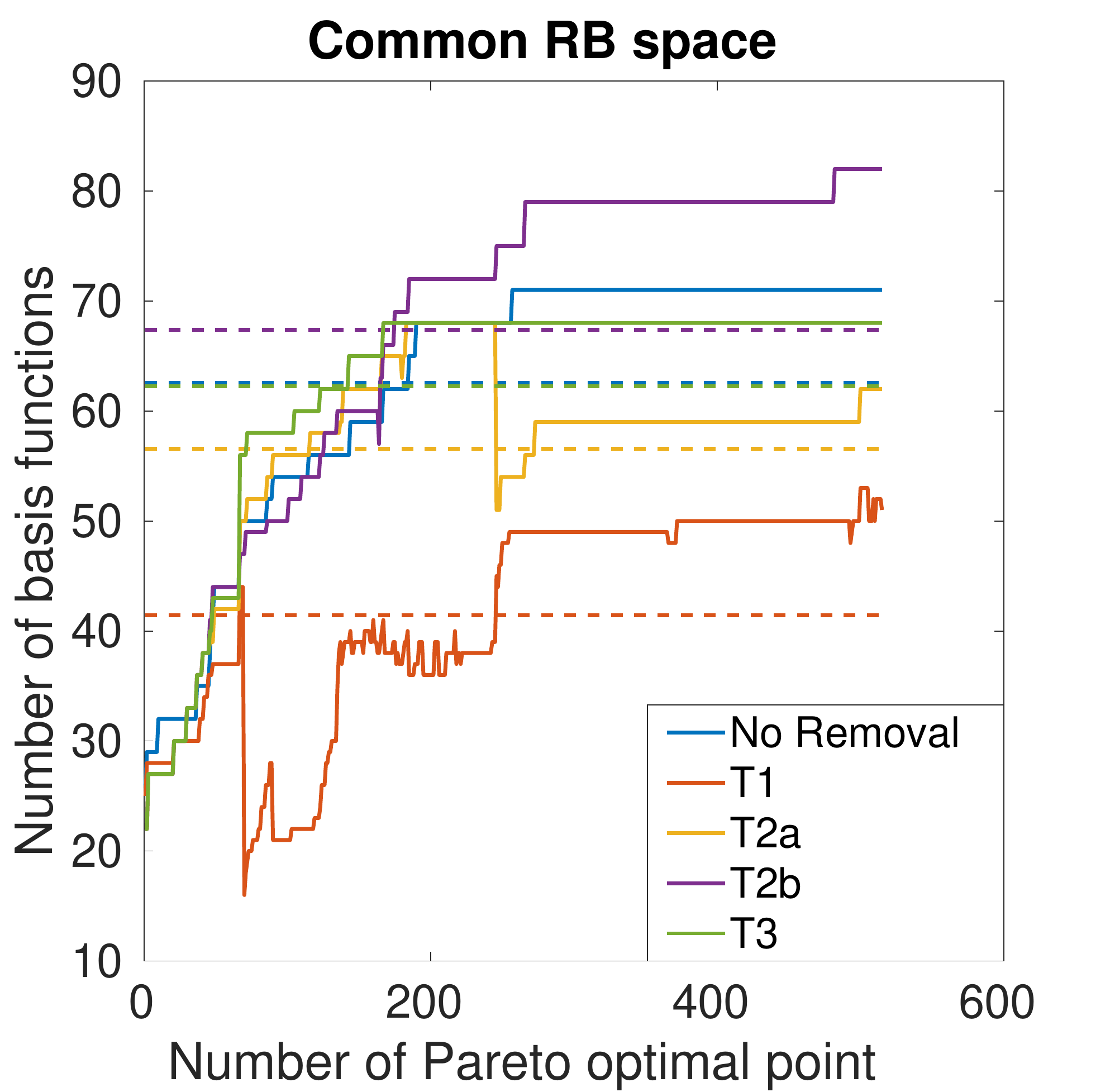}}
	\caption{Number of basis functions used to compute each Pareto optimal point. (\textbf{a}) Local RB space. (\textbf{b}) Common RB space. Dashed lines: average number of basis functions.\label{Fig:Numberofbasis}}
\end{figure}
In conclusion, comparing our fastest method (i.e., Algorithm~\ref{Algorithm:MOR:TRRB} with local RB spaces and T3) to the slowest (i.e., using \cite{Ban21} with a common RB space) we get essentially the same results (the approximation error is $10^{-6}$) with half of the time, which is roughly 500 seconds. This shows how one should invest time and resources in providing efficient techniques for reducing the number of basis functions in the RB space, while using an adaptive TR-RB algorithm. Particularly in the case of multiobjective optimization, this becomes crucial for a large number of cost functionals $k$. To obtain the same resolution of the Pareto front as in Figure~\ref{Fig:ParetoFronts} for a large $k$, we will need to solve the PSPs for many more points, implying higher risk of having redundant basis functions.

\section{Conclusions}
We presented and analyzed novel ways of reducing the dimension of the RB space during the optimization procedure. To our knowledge, this has not been addressed yet for the RB method, although it is common for other model order reduction techniques. Such a removal significantly improved the performances of the TR-RB algorithm in the context of multiobjective optimization, leading faster to an accurate solution than the already existing techniques. These removal techniques can also be extended to other applications in which sequential parametric PDE-constrained optimization problems must be solved. In future work, one can try to achieve further improvements concerning robustness of the method and deriving tighter a-posteriori error estimators, in particular for the gradient of the cost function. This is also of great interest in the RB community.


\section*{Acknowledgments}
The authors acknowledge funding by the Deutsche Forschungsgemeinschaft (DFG) for the project Localized Reduced Basis Methods for PDE-constrained Parameter Optimization under contract VO 1658/6-1. The authors thank Tim Keil, Mario Ohlberger and Felix Schindler from University of M\"unster (Germany) for the fruitful exchange of ideas on the topic.


\end{document}